\documentclass[11pt,reqno]{amsart} 

\usepackage[margin = 1.1in]{geometry}
\usepackage{amsfonts, amsmath,amscd, amssymb, cite, color, latexsym, mathrsfs, mathtools,
slashed, stmaryrd, verbatim, wasysym}
\usepackage[all]{xy}
\usepackage[pdftex]{graphicx}
\usepackage{hyperref}
\usepackage[applemac]{inputenc}
\usepackage[titletoc,title]{appendix}
\usepackage{tikz-cd}
\usepackage{oubraces}
\usepackage{cancel}
\usepackage{float}
\allowdisplaybreaks
\usepackage{tensor}

\usepackage{import}
\usepackage{xifthen}
\usepackage{pdfpages}
\usepackage{transparent}

\DeclareMathAlphabet{\mathboondoxfrak}{U}{BOONDOX-frak}{m}{n}

\usepackage{tikz}
\usetikzlibrary{decorations.pathmorphing, matrix, calc, arrows}

\newlength{\LETTERheight}
\AtBeginDocument{\settoheight{\LETTERheight}{I}}

\newtheorem{theorem}{Theorem}
\newtheorem{corollary}[theorem]{Corollary}

\newtheorem{lemma}[theorem]{Lemma}
\newtheorem{proposition}[theorem]{Proposition}
\newtheorem{remark}{Remark}

\numberwithin{equation}{section}
\numberwithin{theorem}{section}

\newcommand{\bhs}[1]{\mathfrak B_{#1}}

\renewcommand{\bar}{\overline}
\renewcommand{\Re}{\operatorname{Re}}
\renewcommand{\hat}[1]{\widehat{#1}}

\newcommand{\wt}[1]{\widetilde{#1}}

\newcommand{\wf}{\text{WF}}

\newcommand{\wfp}{\text{WF}^\prime}

\newcommand{\Rp}{\mathbb{R}^+}

\newcommand\pa{\partial}

\newcommand\eps\varepsilon
\renewcommand\epsilon\varepsilon

\newcommand\CI{{\mathcal{C}}^{\infty}}

\newcommand\diag{\operatorname{diag}}

\newcommand\ev{\operatorname{even}}

\newcommand\Id{\operatorname{Id}}
\renewcommand\Im{\operatorname{Im}}

\newcommand\Ker{\operatorname{Ker}}

\renewcommand\Re{\operatorname{Re}}

\newcommand\Tr{\operatorname{Tr}}

\newcommand\Vol{\operatorname{Vol}}

\newcommand\paperintro%
        {%
         }
\newcommand\paperbody%
        {%
         }

\newcommand\bbC{\mathbb{C}}

\newcommand\bbH{\mathbb{H}}

\newcommand\bbN{\mathbb{N}}

\newcommand\bbR{\mathbb{R}}
\newcommand\bbS{\mathbb{S}}

\newcommand\bbX{\mathbb{X}}

\newcommand\cD{\mathcal{D}}
\newcommand\cE{\mathcal{E}}

\newcommand\cK{\mathcal{K}}

\newcommand\cO{\mathcal{O}}

\newcommand\cR{\mathcal{R}}

\newcommand\cV{\mathcal{V}}

\newcommand\sH{\mathscr{H}}

\newcommand\sV{\mathscr{V}}

\DeclareMathAlphabet{\mathpzc}{OT1}{pzc}{m}{it}

\title[The Wave Kernel on ACH Manifolds]{The Wave Kernel on Asymptotically Complex Hyperbolic Manifolds}

\author{Hadrian Quan}

\begin{document}

\maketitle

\begin{abstract}
We study the behavior of the wave kernel of the Laplacian on asymptotically complex hyperbolic manifolds for finite times. We show that the wave kernel on such manifolds belongs to an appropriate class of Fourier integral operators and analyze its trace. This construction proves that the singularities of its trace are contained in the set of lengths of closed geodesics and we obtain an asymptotic expansion for the trace at time zero. 
\end{abstract}

\section{Introduction}

There is a long-standing research program investigating the spectral and scattering theory of \emph{real} asymptotically hyperbolic manifolds, see e.g. \cite{And, Ale-Maz, Ale-Ba-Na, CDLS, Feff-Gra, Gra-Wit, Gra-Zwo, Jo-Sa1, Va} and references contained therein, for a small sample of the surrounding work. However there is comparatively much less work concerning the analogous setting of \emph{asymptotically complex hyperbolic manifolds}. 
These spaces were first introduced by Epstein, Mendoza, and Melrose \cite{EMM}, and more recently have been investigated extensively by \cite{Gu-SaBr,Feff-Hi, Gu-SaBr, HMM, PHT, Yo1, Yo2}. This class of manifolds includes certain quotients of complex hyperbolic space by discrete groups, as well as strictly pseudoconvex domains in Stein manifolds equipped with K\"{a}hler metrics of Bergman type. 

In this work we extend major results which study the wave kernel of asymptotically real hyperbolic manifolds to this complex setting.  Joshi-S\'{a} Barreto \cite{Jo-Sa2} study the wave kernel by exhibiting this operator as an element of a certain algebra of Fourier integral operators which have been adapted to the geometry at infinity of this class of real asymptotically hyperbolic manifolds. In the case of both works, moving from the real to the complex case presents new difficulties to the analysis. On the other hand, the original methods of both Vasy and Joshi-S\'{a} Barreto are robust enough to permit an analysis of this class of manifolds of hyperbolic-type. 

Before introducing the structure of complex hyperbolic manifolds we briefly recall the geometry of real asymptotically hyperbolic manifolds. A non-compact Riemannian manifold $(M,g)$ of real dimension $(n+1)$ is called \emph{asymptotically hyperbolic} if it compactifies to a $\CI$ manifold $\bar M$ with compact boundary $\pa \bar M$, equipped with a boundary defining function $\rho$, and such that $\rho^2 g$ is a $\CI$ metric which is non-degenerate up to the boundary, and moreover that $|d\rho|_{\rho^2 g}^2\equiv 1$ at $\pa \bar M$. This name is due to the fact that the final hypothesis ensures that along any smooth curve in $\bar M\setminus \pa \bar M$ approaching a point in $\pa\bar M$, all sectional curvatures of $g$ approach $-1$, see e.g. \cite{MazMel}. 

As proven in \cite{Jo-Sa1}, these geometric hypotheses are equivalent to the existence of a product-type decomposition at infinity $M\sim [0,\eps)_\rho \times \pa M$, such that
\[ g = \frac{d\rho^2 + g_0(\rho)}{\rho^2} , \]
where $g_0(\rho)$ is a $\CI$ 1-parameter family of $\CI$ metrics on $\pa \bar M$. In this model, the boundary $\pa \bar M$ represents the geometric infinity of $\bar M$, analogous to the role played by the $\bbS^n$ at infinity in $\bbH_\bbR^{n+1}$. In particular the metric $\rho^2 g|_{\pa M}$ fixes a conformal representative of a metric on $\pa \bar M$. 

The spectrum of the Laplacian of such manifolds was first studied by \cite{MazMel}; they determined that it is comprised of finitely many $L^2$-eigenvalues $\sigma_{\text{pp}}(\Delta_g)\subset (0,\tfrac{(n+1)^2}{4})$ and the absolutely continuous spectrum $\sigma_{\text{ac}}(\Delta_g)=[\tfrac{(n+1)^2}{4},\infty)$. In particular, they prove that the resolvent
\[ R(\zeta)=(\Delta_g-\zeta(n+1-\zeta))^{-1},  \]
is well-defined as a bounded operator on $L_g^2(X)$ whenever $\Re(\zeta)>\tfrac{n+1}{2}$. Further they prove that $R(\zeta)$ has a meromorphic extension to $\bbC\setminus \tfrac12((n+1)-\bbN_0)$, as an operator $R(\zeta):\CI_0(X)\to \CI(X)$, and with only finite order poles (this extension is meromorphic on the whole complex plane assuming the metric is \emph{even} in the sense of \cite{Gui}). 

We now move to introducing the complex analogue of these spaces, and introduce our results. We say a non-compact Riemannian manifold $(X,g)$, of \emph{complex} dimension $(n+1)$, is an asymptotically complex hyperbolic manifold (hereafter ACH manifold) if the following holds. We assume $X$ compactifies to a $\CI$ manifold $\bar X$, compact with boundary, equipped  with a choice of boundary defining function $r$ (hereafter, a bdf). This is a smooth nonnegative function on $\bar X$ which such that 
\[ \bar X=\{r=0\}, \quad dr|_{\pa \bar X}\neq 0.  \]

We further assume the boundary admits: (1) a contact form $\theta\in \Omega^1(\pa\bar X)$ defined as satisfying $\theta\wedge (d\theta)^{n}\neq 0$; (2) an almost complex structure $J:\Ker \theta\to \Ker\theta$; such that $d\theta(\cdot,J\cdot)$ is a symmetric, positive-definite bilinear form on $\Ker\theta$. Then we say $(X,g)$ is an ACH manifold if there is a tubular neighborhood $\Phi: U\to \pa \bar X\times [0,\eps)_r$ of the boundary $\pa \bar X$ such that 
\begin{equation}
g \sim \Phi^*g_\theta \; \text{ as $r \to 0$}, \hspace{6mm} g_\theta = \left(\frac{4dr^2}{r^2}+\frac{d\theta(\cdot,J\cdot)}{r^2} + \frac{\theta^2}{r^4}\right) = \frac{4dr^2 + g_0(r)}{r^2} .\end{equation}
In particular, for another choice of boundary defining function, $\wt r$, we observe that $r^4 g|_{\pa \bar X}=e^{4 f}\theta$, for some $f\in \CI(\bar X)$. Denoting the conformal class of our contact structure by $[\theta]$ we can consider the boundary as being endowed with the structure of a \emph{conformal pseudohermitian manifold} $(\pa \bar X,[\theta], J)$. This is analogous to the natural conformal structure on $(\pa \bar M, [\rho^2 g])$ in the real hyperbolic case. 

Before continuing, we require an additional hypothesis, which is that $g$ is an \emph{even metric}; i.e., the dual metric $g^{-1}$ defined on $T^* \bar X$ has only even powers of $r$ in a Taylor expansion at $r=0$. This is automatic in the case of $\bbH_\bbC^{n+1}$, and necessary for the existence of a meromorphic continuation of the resolvent of $\Delta_g$ to all of $\bbC$, (in fact, the failure of this hypothesis implies the existence of at least one essential singularity in the continuation of the resolvent, see \cite{Gui}, \cite{Gu-SaBr}).

In the case that the metric of $(X,g)$ is even in the above sense, we can replace the smooth structure on this manifold with its \emph{even smooth structure}, denoted $X_{\ev}$. In this case the smooth structure on $X$ has been modified by declaring that only functions which are even in $r$ are smooth with respect to $X_{\ev}$. This change of the smooth structure permits us to define a square root of our original defining function, and guarantee that it is an element of $\CI(X_{\ev})$. Equivalently, the even smooth structure can be defined by declaring $X_{\ev}$ is a smooth manifold with boundary, with bdf $r^2$. Throughout we shall denote the square root of our bdf $\rho=r^2$.

Now we state our main results on the behavior of solutions to the wave equation for small times. This question can be approached by a study of the fundamental solution to the wave equation, as in the work of Joshi-S\'{a} Barreto \cite{Jo-Sa2} who studied the wave operator \\ $\cos\big(t\sqrt{\Delta_g-(n+1)^2/4}\big)$ in the setting of real asymptotically hyperbolic manifolds. This operator has Schwartz kernel $U(t,p,p^\prime)$ satisfying
\[ \begin{cases} \left(\pa_t^2+\Delta_g-\tfrac{(n+1)^2}{4}\right)U(t,p,p^\prime)=0 \\ U(0,p,p^\prime)=\delta(p,p^\prime), \quad \pa_t U(0,p,p^\prime)=0  \end{cases},  \]
and they prove that $\cos\big(t\sqrt{\Delta_g-(n+1)^2/4}\big)$ resides in an algebra of Fourier integral operators. Having shown this, they use the results of \cite{Du-Gu, Hor1,Hor3} to study its (regularized) trace. 

This construction of the wave group $U(t,p,p')$ as a Fourier integral operator was motivated by the analysis of the resolvent of a real asymptotically hyperbolic manifold initiated in \cite{MazMel}. Mazzeo-Melrose obtained their results by exhibiting the resolvent as an element of the ``large" calculus of zero pseudodifferential operators $\Psi_0^*(M)$; i.e., those pseudodifferential operators with Schwartz kernels constructed as distributions on the blown-up space $\bar M\times_0\bar M$, obtained by blowing up the intersection of the the corner $\pa \bar M\times \pa \bar M$ with the diagonal $\bar M_{\diag}\hookrightarrow \bar M\times \bar M$ in $\bar M\times \bar M$. The new boundary hypersurface resulting from this blow-up is called the \emph{front face}. (For an extended treatment on such blow-ups see \cite[\S 3]{MazMel}, \cite{Melrose:Corners}, and \cite{Grieser}) 

Along such lines \cite{Jo-Sa2} construct a class of zero Fourier integral operators as those operators whose Schwartz kernels, when lifted to $\bar M\times_0 \bar M$, have support away from the \emph{left} and \emph{right} boundary faces (i.e. the lifts of $\pa \bar M\times \bar M$ and $\bar M\times \pa \bar M$ respectively). This greatly simplifies the construction of this class of operators, as typically the corners formed by the intersections of the left face (resp. right) with the front face would need to be incorporated into the definition of the operators; requiring the support of the Schwartz kernels avoid such corners allows their contributions to be neglected. In particular, due to the finite speed of propagation for the wave equation, a distribution which is initially supported only on the front face (such as $U(t,p,p^\prime)$) will remain supported in the interior of the front face for all finite time. Thus \cite{Jo-Sa2} can construct a small time parametrix for the wave group while remaining entirely in this restricted calculus of zero Fourier integral operators. 

Following this strategy we begin with the notion of the $\Theta$-stretched product, $\bar X\times_\Theta \bar X$, which is the analogous blow-up of the double space $\bar X\times \bar X$ defining the class of $\Theta$-pseudodifferential operators $\Psi_\Theta^*(X)$ used in the study of the resolvent initiated by \cite{EMM}. With the appropriate definition of $\Theta$-Fourier integral operators, we can quickly conclude:

\begin{theorem}$ $\\
Let $G$ be the length functional on $T^*X$, (i.e. the dual metric). For each $t\in \bbR$, the graph of the time-$t$ flow-out of the diagonal in $T^*X\times T^*X$ by the Hamilton vector field $H_G$ is a canonical relation, denoted $C$. Furthermore, the wave group $U(t)$ is a $\Theta$-FIO with respect to this canonical relation.
\end{theorem}

Once we know the wave group is a $\Theta$-Fourier integral operator, it is straightforward to use the results of \cite{Du-Gu, Hor3} to analyze the trace of $U(t,p,p^\prime)$. One subtlety is that the trace needs to replaced with a regularized trace, defined using a Hadamard regularization procedure using our choice of bdf $\rho$. Defining the cut-off wave trace, 
\[T_\eps(t) = \int_{\{\rho>\eps\}} U(t,p,p)\]
we obtain
\begin{proposition}\label{cutoff-PR}
There exists $\eps_0>0$ such that for all $\eps<\eps_0$, the singular support of $T_\eps$ is contained in the set of periods of closed geodesics of $X$.
\end{proposition}

With this result in hand, after choosing a smooth cutoff $\chi(t)\in \CI_0(\bbR)$ supported away from the lengths of all non-zero periods of closed geodesics, and using the results of \cite{Hor1} we obtain a Duistermaat-Guillemin type result for the cutoff wave trace.

\begin{theorem}\label{wavetrace-simple}
There exists $\{\omega_k\}_{k\in \bbN_0}\subset \bbR$ such that the renormalized trace ${}^{R}\Tr U(t)$ satisfies,
\begin{equation*}
\int_\bbR {}^{R}\Tr U(t) \chi(t) e^{t\mu}  dt \sim \frac{1}{(2\pi)^{2n+2}} \sum_{k=0}^\infty \omega_k \mu^{2n+2-2k},  \end{equation*}
as $\mu\to 0$ and is rapidly decaying as $\mu\to -\infty$. The leading term, $\omega_0={}^{R}\Vol_g(X)$, is called the renormalized volume, and can be computed as
\begin{equation}\label{renorm-vol} 
{}^{R}\Vol_g(X) = \lim_{\eps\to 0} \left[ \int_{\{\rho>\eps\}} d\Vol_g - \sum_{j=-2n-2}^{-1} d_j \eps^j - d_0\log(1/\eps) \right] ,  \end{equation}
where $d_j$ are the unique real numbers such that this limit exists. 
\end{theorem}

Finally, we remark on the appearance of the renormalized volume in Theorem \eqref{wavetrace-simple}. In the real hyperbolic setting it is known that the renormalized volume is, in certain dimensions, independent on the choice of representative of the conformal infinity. Namely, for $(M^{n+1},g)$ a real asymptotically hyperbolic manifold, one can similarly define the renormalized volume as the finite part of the in the expansion of $\Vol_g(\{x\geq \eps\})$ as $\eps\to 0$, given a choice of bdf $x$. For $n$ odd, the real hyperbolic renormalized volume is independent of $h_0$, the choice of conformal representative. On the other hand, for $n$ even, we suddenly have the dependence of the renormalized volume on this choice of representative of $[h]$. This is result is the so-called holographic anomaly (see \cite{skenderis}) and motivates much of the interest of asymptotically hyperbolic manifolds in mathematical physics, for their connection with the anti deSitter/conformal field theory (AdS/CFT) correspondence. 

More concretely, the volume expansion of $(M^{n+1}, g)$ of an \emph{Einstein} asymptotically hyperbolic metric, for $n$ even, is given by
\begin{equation*}
\Vol_g(\{x\geq \eps\}) = V_{-n}\eps^{-n}+V_{-n+2}\eps^{-n+2}+\ldots + V_{-2}\eps^{-2} + V_0\log(1/\eps) + {}^{R}\Vol_g(M) +o(1) ,  
\end{equation*}
and \cite{Gra-Zwo} first made the connection of $V_0$ to Branson's $Q$-curvature \cite{Branson},
\[ V_0 = 2 c_{n/2} \int_{\pa M} Q,   \]
for $c_{n/2}$ a dimensional constant. In the ACH setting, the renormalized volume was first studied at this level of generality by Matsumoto in \cite{Yo1}. Our construction of the renormalized wave trace thus provides an alternate proof of Matsumoto's result, via formula \eqref{renorm-vol}. For a general ACH metric, \cite{Yo1} generalizes this result for an analogue of Bransons $Q$-curvature. From his result we obtain as a corollary that the constant $d_0$ in our Theorem \ref{wavetrace-simple}, is given by,
\[ d_0 = \frac{2(-1)^{n+1}}{n!^2(n+1)!} \int_{\pa X} Q_\theta^g \theta\wedge (d\theta)^n .  \]
This quantity is a global CR invariant of the boundary, thus leading to a pseudoconformal analogue of the holographic anomaly. Given these results there is strong connection between the renormalized volume of an ACH manifold and its spectrum. On the mathematical physics side there seems to be relatively scarce work on this complex analogue of the AdS/CFT correspondence. 

$ $\\
{\bf Funding}
This work was supported by the National Science Foundation grant numbers DGE-1746047, DMS-1440140, and DMS-1711325, growing out of conversations while the author was in residence at the Mathematical Sciences Research Institute in Berkeley, California during the Fall 2019 semester. 

$ $\\
{\bf Acknowledgements.} 
The author is very grateful to acknowledge the kind discussions and suggestions of Andras Vasy, Ant\^onio S\'a Barreto, and Pierre Albin.

\tableofcontents

\section{The geometry of asymptotically complex hyperbolic manifolds}\label{sec:geom}

Because the construction of our adapted FIO-calculus entails a finer understanding of the geometry of an asymptotically complex hyperbolic manifold, we briefly recall the geometry of the Bergman-type metric our manifold is endowed with. 

Let $(X,\pa X)$ be a non-compact manifold with closed boundary. We assume the boundary admits a contact form $\theta$ and an almost complex structure $J:\Ker(\theta)\to \Ker(\theta)$ (i.e., an endomorphism satisfying $J\circ J=-\text{Id}_{\Ker(\theta)}$) such that $d\theta(\cdot, J\cdot)$ is symmetric positive definite on $\Ker(\theta)$. We consider a metric $g_{ACH}$ of the following form: there is a boundary defining function $\rho$,
\[ \pa X=\{\rho=0\}, \quad d\rho\rvert_{\pa X}\neq 0  \]
such that in a collar neighborhood $\varphi: [0,1)_\rho \times \pa X_{\omega,z}  \to U$ it takes the form
\begin{equation}\label{Theta-metric} \varphi^*g_{ACH} = \frac{d\rho^2}{\rho^2} + \frac{ d\theta(\cdot, J\cdot) }{\rho^2} + \frac{\theta\otimes\theta}{\rho^4} + \rho Q_\rho = \frac{d\rho^2 + h(\rho,\omega,z, d\omega, dz)}{\rho^2} ,  \end{equation}
where $(\cD_\rho^\bbH)^*Q_\rho$ is a smooth section of $S^2(T^*X)\cap \Ker(\iota_{\pa_\rho})$. Here, $\cD_\rho^{\bbH}$ denotes the anisotropic dilation map
\begin{equation*}
	T_q\pa X = \sH_q \oplus \sV_q \ni (v_H, v_V) \xmapsto{\phantom{x}\cD_{\rho}^{\bbH}\phantom{x}} (\rho v_H, \rho^2v_V) \in \sH_q \oplus \sV_q = T_q\pa X,
\end{equation*}
with splitting induced by the choice of contact structure $(\pa X, \theta)$, (i.e., $\sH=\Ker \theta$).

We observe that for any other choice of defining function $\wt \rho$ we have
\[ \wt\rho^4 g\rvert_{\pa X} = e^{4\omega_0} \theta\otimes \theta , \text{ for some }\omega_0\in \CI(X),  \]
thus it is more natural to associate to $g_{ACH}$ a conformal class of 1-forms $[\Theta]$.  The boundary manifold equipped with the data of $(\pa X,\theta,J)$ is a \emph{closed pseudohermitian manifold}. The corresponding \emph{conformal pseudohermitian structure} $([\Theta],J)$ was called a $\Theta$-structure in \cite{EMM}.

This Riemannian metric structure describes a non-compact incomplete manifold whose metric is asymptotic to complex hyperbolic space $\bbH_{\bbC}^{n+1}$. A useful model of complex hyperbolic space $\bbH_\bbC^{n+1}$ is given by
\[ \bbH_\bbC^{n+1} = \{ \zeta\in \bbC^{n+1} : Q(\zeta,\zeta)>0 \}, \quad \text{where} \quad Q(\zeta,\zeta')=-\tfrac{i}{2}(\zeta_1-\zeta_1')-\tfrac12 \sum_{j>1} \zeta_j\bar{\zeta_j}'  \]
with boundary sphere equal to a compactification of the $(2n+1)$-dimensional Heisenberg group, 
\[ H_n:=\{ \zeta\in \bbC^{n+1}: Q(\zeta,\zeta)=0\} = \{ (\zeta_1,w)\in \bbC^{n+1}: \tfrac12|w|^2=\Im(\zeta_1)  \} \simeq \bbC^n \times \bbR. \]
This model of complex hyperbolic space realizes $\bbH_{\bbC}^{n+1} \simeq \Rp\times H_n$ 
with the coordinates 
\[ \rho(\zeta) = Q(\zeta,\zeta)^{1/2}, \;\; w \in \bbC^n, \;\; z=\Re(\zeta_1),  \]
foliating $\bbH_{\bbC}^{n+1}$ by a family of $H_n$-hypersurfaces. Writing $w=x+iy$, in these coordinates we can also write the contact form at the boundary as
\[ \theta_0 = dz + \sum_{j=1}^n y_jdx^j - x_jdy^j ,  \]
and the metric on complex hyperbolic is the Bergman metric,
\[  g_{\text{Berg}} = \frac{4d\rho^2 + 2|dw|^2}{\rho^2} + \frac{\theta_0^2}{\rho^4}.  \]

The Heisenberg group is a Lie group of dimension $2n+1$. In these coordinates the group law is given by
\[ (x,y,z)\cdot_H (x',y',z') = (x+x',y+y', z+z' + \Im[(x,y)\cdot (\overline{x',y'})] ), \]
which is abelian in the first $2n$ components. Its Lie algebra $\mathfrak{h}$ has a basis $\{ X_j,Y_j, Z \}$, which satisfies the non-trivial bracket relations: $[X_j,Y_j]=Z$ for all $j=1,\ldots, n$ and all brackets vanishing. This structure of a nilpotent Lie algebra gives an identification $H_n\to \mathfrak{h}$ of the form
\[  \bbC^n\times \bbR \; \ni \; ((x,y),z)  \mapsto \sum_{j=1}^n x_j X_j+y_jY_j +zZ \; \in\; T_{\{e\}}H_n , \]
after which the group law can be written via Lie algebra elements $W,W^\prime\in \mathfrak{h}$ as,
\[ W\cdot_H W' = \pi_{\Ker \theta_0}(W+W') + (\theta_0(W+W') - d\theta_0(W,W'))Z.  \]
It is a consequence the nilpotence of $H_n$ that the group law is a finite order polynomial in the Lie algebra elements, rather than the asymptotic series given by the Baker-Campbell-Hausdorff formula, (see e.g. \cite{eich}).  

Finally, we explain how the complexified hyperbolic space arises as a semi-direct product: there is parabolic dilation on $H_n$ (consistent with the bracket relations of the Lie algebra $\mathfrak{h}$) given by $\cD_\delta(x,y,z) = (\delta x,\delta y,\delta^2 z)$. The group law on the semidirect product $\bbH_{\bbC}^{n+1} \simeq \Rp\rtimes_{\cD_\delta} H_n$ is given as
\begin{equation}\label{LAlg-group-law} (\rho, W)\cdot_{\bbH_{\bbC}} (\rho',W') = (\rho\rho', W\cdot_{H} \cD_{\rho}(W')). \end{equation}
The geometric picture described above of complex hyperbolic space being foliated by a family of Heisenberg groups as level-set hypersurfaces of $\rho$ is compatible with this group law: an open set in $\{\rho = c\}\simeq H_n$ is related to the corresponding set in $\{\rho = c+\eps\}$ by pullback along $M_\eps$. 

Our reason for expressing the Lie group law of $\bbH_{\bbC}^{m+1}$ at the level of its Lie algebra is that the Lie algebra arises more naturally at the level of tangent spaces in our later analysis.\\

\section{The wave kernel on asymptotically complex hyperbolic manifolds}\label{sec:waveker}

In this section we begin the construction of a Fourier Integral Operator Calculus, which is adapted to the asymptotic geometry of the metric \eqref{Theta-metric}. Such a calculus will be comprised of operators whose Schwartz kernels have prescribed asymptotics on a manifold with corners, the $\Theta$-stretched product $X\times_\Theta X$ of \cite{EMM}. 

Analogously to the 0-blow up, Epstein-Mendoza-Melrose defined the $\Theta$-blow up of an ACH manifold; this will be very similar to the zero-blow up of an AH manifold. The biggest distinction being the blow-up at the front face is non-isotropic, reflecting the different asymptotics in $\rho$ of boundary vector fields (namely those vector fields whose $g_{ACH}$-duals span $d\theta(\cdot , J\cdot )$ vs $\theta\otimes \theta$).

Following \cite{EMM}, we next explain how we will modify the product $X\times X$ to construct our algebra of Fourier integral operators. We begin with the notion of the $\Theta$-vector fields $\cV_\Theta$:
\[ V\in \cV_\Theta \iff V\in \rho \cdot \CI(X; TX), \quad \wt\theta(V)\in \rho^2 \cdot \CI(X;TX), 
\]
where $\wt \theta\in \CI(X; TX)$ is any smooth extension of $\theta$ to all of $X$. It is shown in \cite[\S 1]{EMM} that this definition is dependent only on the choice of conformal class of $[\theta]$. This is partly because a representative of $[\theta]$ determines a local frame by requiring
\begin{equation}\label{ONF} \{ X_1,\ldots, X_n, Y_1,\ldots, Y_n \} \text{ is an orthonormal frame of } d\theta(\cdot, J\cdot), \hspace{4mm} \theta(Z)=1, \;\; \theta(\pa_\rho)=0 , \end{equation}
in which we can express
\[ \cV_\Theta = \text{span}_{\CI}\{ \rho\pa_\rho, \rho X_1,\ldots, \rho X_n, \rho Y_1,\ldots, \rho Y_n, \rho^2 Z\} ,  \]
and a different choice of bdf $\rho'$ produces a frame as in \eqref{ONF} associated to a contact form $ \theta'$ conformal to $\theta$. 

Given this $\CI(X)$-module, we can define the $\Theta$-tangent bundle ${}^{\Theta}TX$. This is a vector bundle over $X$, with a bundle map $\iota_\Theta: {}^{\Theta}TX\to TX$, which is an isomorphism over $X\setminus \pa X$ such that
\[ \CI(X; {}^{\Theta}TX) = \iota_\Theta^*(\cV_\Theta).  \]
Next, we construct the $\Theta$-stretched product of \cite[\S 8]{EMM}. Notice first that in the product $X\times X$, the boundary of the diagonal $\pa X_{\diag}\simeq \pa X$ is an embedded submanifold,
\[ \pa X_{\diag}\hookrightarrow \pa X\times \pa X\hookrightarrow X\times X  \]
and is a \emph{clean submanifold} in the sense of \cite{Du-Gu}, since it is an embedded submanifold of the corner, and thus all differentials of bdfs vanish at $\pa X_{\diag}$. 
The 1-form $\theta$ on $X$ defines a line subbundle 
\[ \sH^* \subset N_{X\times X}^*(\pa X_{\diag}) \]
spanned by 
\[  \pi_L^*\theta - \pi_R^*\theta, \] 
with $\pi_{(\cdot)}:X\times X\to X$ denoting the projection onto the left and right factors respectively.
\begin{figure}[H]\label{theta-blowup}
\includegraphics[scale=.45]{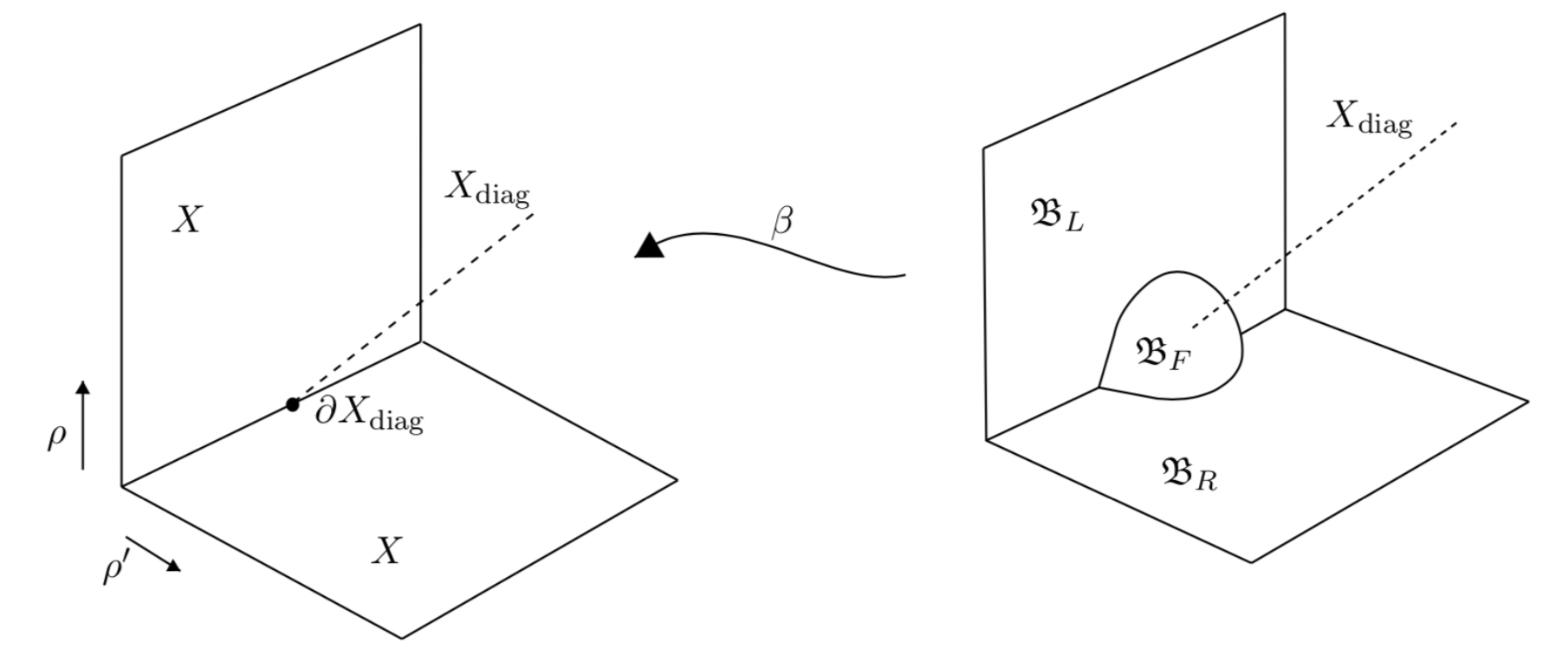}
\caption{The blow-down map $\beta$ of the $\Theta$-stretched product space $X\times_\Theta X$}
\end{figure}
With this trivialization of the conormal bundle, we define the $\Theta$-blow up of the corner as the $\sH^*$-parabolic blow-up (defined using the dilation structure on fibers given in \eqref{LAlg-group-law}) of the boundary diagonal:
\begin{equation*}
\begin{gathered} 
X\times_{\Theta} X = [ X\times X; \pa X_{\diag} , \sH^* ]  := (X\times X \setminus \pa X_{\diag}) \sqcup \bbS N_{\sH,+}(\pa X_{\diag})  \\
\bbS N_{\sH,+}(\pa X_{\diag}) = (N_{X\times X}\pa X_{\diag})/\Rp_{\sim_{\cD^\sH}}
\end{gathered}
\end{equation*}
where the equivalence on fibers $\cD^\sH$ is defined using the decomposition $ N(\pa X_{\diag})_+ = \sH \oplus \sH^\perp,$ with $\sH=\text{Ann}(\sH^*)$,
\[ (W, Z)\sim_{\cD_\delta^{\sH}} (W^\prime, Z^\prime)\iff \exists \delta>0, \; (W,Z) = ( \delta W^\prime, \delta^2 Z^\prime) .   \]
This real unoriented blow-up replaces the submanifold $\pa X_{\diag}$ with its inward-pointing \emph{parabolic}-sphere bundle. This blow-up procedure furnishes a blow-down map
\[ \beta: X\times_{\Theta} X\to X\times X,  \]
which is the identity on $X\times X \setminus \pa X_{\diag}$, and given by the bundle projection map of the parabolic-sphere bundle on $\bbS N_{\sH,+}(\pa X_{\diag})$. This is a manifold with corners, and has three new boundary faces:
\begin{equation*}
\begin{gathered}
\bhs{F} = \beta^{-1}(\pa X_{\diag}) = \bbS N_{\sH,+}(\pa X_{\diag})  \\
\bhs{L} = \overline{ \beta^{-1} \{(X\times \pa X)\setminus \pa X_{\diag}\} }, \quad \quad \bhs{R} = \overline{\beta^{-1}\{ (\pa X\times X)\setminus \pa X_{\diag} \} }. 
\end{gathered}
\end{equation*}
By construction, the front face $\bhs{F}$ is a fiber bundle over $\pa X_{\diag}$ with fiber a projective quotient of the inward pointing normal bundle $N_{X\times X}(\pa X_{\diag})_+$; the front face has fiber over $p\in \pa X$ given by
\begin{align*} 
\bhs{F}|_p = [N_{X\times X}(\pa X_{\diag})_+ & \setminus \pa X_{\diag}] / \sim_{\cD_\delta^\sH} . 
\end{align*}
For more details and the proof of diffeomorphism invariance of this construction see \cite[\S 5-7]{EMM}.

\subsection{The $\Theta$-symplectic structure on ${}^{\Theta}T^*X$}$ $

Similarly as in the \cite{Jo-Sa2}, associated to the Lie algebra $\cV_\Theta$ we can define the notion of a $\Theta$-Fourier integral operator, which will be operators whose Schwartz kernels have prescribed asymptotics on a resolution of the product $X\times X$, the $\Theta$-stretched product $X\times_\Theta X$. A standard Fourier integral operator is characterized by its Schwartz kernel having singular support conormal to a Lagrangian inside $(T^*X\setminus o)\times(T^*X\setminus o)$; to generalize this notion we must first understand how Lagrangians arise in the symplectic structure of ${}^{\Theta}T^*X$.

In a neighborhood of the boundary $U$, if we use coordinates $(x, \zeta)=((\rho,w,z); (\xi,\eta_H, \eta_V))\in {}^{\Theta}T^*X$, where $\pa X=\{\rho=0\}$, then 
\[ \theta|_{U\cap \{\rho=0\}} = dz - \tfrac12\sum_{j=1}^{n} w_{x_j}dw_{y_j} - w_{y_j}dw_{x_j}  \]
the induced map on the dual bundles is given by 
\begin{equation}\label{iota-bar-theta} 
\bar\iota_\Theta:T^*X\to {}^{\Theta}T^*X, \;\; ((\rho,w,z); (\xi,\eta_H,\eta_V))\mapsto ((\rho,w,z) ; (\rho\xi,\rho\eta_H,\rho^2\eta_V)) =: ((\rho,w,z);(\mu,u,t)).  
\end{equation}
In these coordinates the canonical 1-form
\[ \alpha = \xi d\rho + \eta_H\cdot dw + \eta_V dz   \]
pulls back  to the 1-form
\[
\bar\iota_\Theta \alpha = {}^{\Theta}\alpha = \frac{\mu}{\rho}d\rho + \frac{u}{\rho} dw + \frac{t}{\rho^2} dz, 
\]
and hence we have a symplectic form,
\begin{equation}\label{symp-form-coord}
{}^{\Theta}\omega = d({}^{\Theta}\alpha) = \tfrac{1}{\rho}d\mu\wedge d\rho+\tfrac{1}{\rho}du\wedge d\omega + \tfrac{1}{\rho^2} \left( dt\wedge dz  - d\rho \wedge(u d\omega) \right) -\tfrac{2}{\rho^3} d\rho\wedge (tdz). 
\end{equation}
With this symplectic structure on ${}^{\Theta}T^*X$ we can explore the many ways to create Lagrangian submanifolds on this rescaled bundle.

Following \cite{Jo-Sa2}, we can define extendible Lagrangian submanifolds. Set
\[ (X\times_\Theta X)_d=X\times_\Theta X \bigsqcup_{\bhs{F}} X\times_\Theta X,\]
the double of the $\Theta$-stretched product across the front face. We say that a smooth conic closed Lagrangian submanifold $\Lambda \subset T^*(X\times_\Theta X)$ is \textbf{extendible}, if it intersects $T^*(X\times_\Theta X)\rvert_{\bhs{F}}$ transversely. This implies there exists a smooth conic Lagrangian $\Lambda_{\emph{ext}}\subset T^*(X\times_\Theta X)_d$ such that
\[ \Lambda= \Lambda_{\emph{ext}}\cap T^*(X\times_\Theta X), \quad  \Lambda_\Theta: = \Lambda\pitchfork T^*(X\times_\Theta X)\rvert_{\bhs{F}}  \]

One reason for the interest in extendible Lagrangians is that their intersection with the cotangent bundle over the front face is again a Lagrangian submanifold. 
\begin{lemma}\label{lem:extlagr}
If $\Lambda\subset T^*(X\times_\Theta X)$ is extendible then $\Lambda_\Theta=\Lambda\cap T^*(X\times_\Theta X)\rvert_{\bhs{F}}$ is a Lagrangian submanifold of $T^*\bhs{F}$
\end{lemma}
\begin{proof}
Fix coordinates $(\rho,w_1,\ldots,w_{2n},z)$ of $X\times_\Theta X$ valid near $\bhs{F}=\{\rho=0\}$, and with dual variables $(\xi, \eta_H^1,\ldots, \eta_H^{2n}, \eta_V)$. Then $(\rho,w,z; \xi, \eta_H,\eta_V)$ give local coordinates for $T^*(X\times_\Theta X)$ near $\bhs{F}$. By transversality, $d\rho \rvert_{\Lambda}\neq 0$, thus $\rho$ and some subset of $(w,z; \xi, \eta_H,\eta_V)$ must give local coordinates for $\Lambda$. Since $\Lambda$ is Lagrangian, the canonical 2-form 
\[  \omega_{T^*(X\times_\Theta X)} = d\rho \wedge d\xi + \sum_{j=1}^{2n} dw^j\wedge d\eta_H^j +dz\wedge d\eta_V
\]
must vanish on $\Lambda$; hence it vanishes on $\Lambda_\Theta$ as well. From the overall vanishing of this symplectic form, and the non-vanishing of $d\rho$ on $\Lambda$, we must have that $d\xi$ restricted to $T\Lambda\rvert_{\Lambda_\Theta}$ is a multiple of $d\rho$. This implies existence of a function $\phi(\rho,w, z; \eta_H^j, \eta_V)$ satisfying
\[ \Lambda \subset \{\xi=\rho \phi(\rho,w, z; \eta_H^j, \eta_V)\}.  \]
In particular, $\xi |_{\Lambda_\Theta}=0$ and $\sum dw_j\wedge d\eta_H^j +dz\wedge d\eta_V=0$ on $T\Lambda_\Theta$. 
\end{proof}

Having introduced extendible Lagrangians we immediately explain their relation to our the class of distributions we will ultimately be concerned with. We define a Lagrangian distribution associated to an extendible Lagrangian, (either $\Lambda\subset T^*(X\times_\Theta X)$ or $\Lambda\subset T^*\bbR\times T^*(X\times_\Theta X)$), to be the restriction to $X\times_\Theta X$ of a distribution which is Lagrangian with respect to an extension $\Lambda_{ext}$ of $\Lambda$ across $\bhs{F}$. As usual we denote the set of order $m$ distributions which are Lagrangian with respect to $\Lambda$ by $I^m(X\times_\Theta X; \Lambda, {}^{\Theta}\Omega^{1/2})$ (resp. $I^m(\bbR\times X\times_\Theta X; \Lambda, {}^{\Theta}\Omega^{1/2})$).  \\

Now that we have introduced Lagrangians in this setting we can see some ways they arise naturally. If $X,Y$ are two ACH manifolds, a \textbf{$\Theta$-canonical relation} between them is a $\CI$-map 
\[ \chi: \Gamma\subset {}^{\Theta}T^*X \to {}^{\Theta}T^*Y\] 
defined on an open conic subset $\Gamma\subset {}^{\Theta}T^*X$ such that $\chi^* ({}^{\Theta}\alpha_Y) = {}^{\Theta}\alpha_X$. Certain $\Theta$-canonical relations will define Lagrangian submanifolds in $T^*(X\times_\Theta X)$, by associating to $\chi$ its graph relation
\[ \chi: {}^{\Theta}T^*X\to {}^{\Theta}T^*X \leftrightsquigarrow \text{Gr}(\chi) \subset {}^{\Theta}T^*X \times {}^{\Theta}T^* X \]
and we denote such Lagrangians by $\Lambda_\chi$. Particularly relevant Lagrangians will arise from \textbf{liftable canonical transformations}; these are homogeneous canonical transformations $\chi: {}^{\Theta}T^*X\to {}^{\Theta}T^*X$, whose projections to the base is the identity over $\pa X$. 

Using the left and right projections we can define a symplectic form on ${}^{\Theta}T^*X\times {}^{\Theta}T^*X$ by
\begin{equation}\label{difference form} \omega = \pi_1^*\omega_\Theta - \pi_2^*\omega_\Theta .  \end{equation}
Further, the dual to the differential of the blow-down map $\beta: X\times_\Theta X\to X\times X$ induces a smooth map 
\begin{equation}\label{ident} T^* X\times T^*X \simeq T^*(X\times X) \to T^*(X\times_\Theta X) \end{equation}
which is an isomorphism over $\text{Int}(X\times X)$ between $\omega$ and the standard symplectic form on $T^*X\times T^* X$. 

\begin{lemma}{\textbf{(Liftable Canonical Transformations induce Extendible Lagrangians)}}\label{lem:liftlagr} $ $ \\
Let $\chi: {}^{\Theta}T^*X \to {}^{\Theta}T^*X$ be a liftable canonical transformation. The map \eqref{ident}, combined with the identification (over $\text{Int}(X\times X)$) $T^*X\times T^*X\sim {}^{\Theta}T^*X\times {}^{\Theta}T^*X$  gives a smooth map 
\begin{equation}\label{double-ident}  
\varphi_\Theta: {}^{\Theta}T^*X\times {}^{\Theta}T^*X \xrightarrow{\simeq} T^*(X\times_\Theta X) \quad \text{over} \quad \emph{Int}(X\times X)  \end{equation}
which, restricted to the graph of $\chi$, extends by continuity to the boundary and embeds into it as a smooth Lagrangian of $T^*(X\times_\Theta X)$, denoted $\Lambda_\chi$. Further $\Lambda_\chi$ intersects the boundary of $T^*(X\times_\Theta X)$ only over $T^*_{\bhs{F}}(X\times_\Theta X)$, it is extendible across the front face, and this intersection
\[ \Lambda_{\chi_\Theta}:= \Lambda_\chi \cap T^*_{\bhs{F}}(X\times_\Theta X)  \]
defines a Lagrangian submanifold of $T^*\bhs{F}$.
\end{lemma}
\begin{proof}
On the two copies of $X$ in the product $X\times X$ we consider respectively coordinates $(\rho,w,z)$, and $(\rho',w',z')$ valid near the boundary. These induce corresponding local coordinates on the cotangent bundles, which we denote by
\begin{equation}\label{norm-coord}
(\rho,w,z; \xi,\eta_H,\eta_V) \text{ and } (\rho',w',z'; \xi', \eta_H', \eta_V') \text{ corresponding to } T^*X ,  \end{equation}
and 
\begin{equation}\label{theta-coord} 
(\rho,w,z; \mu, u, t) \text{ and }  (\rho' ,w', z';  \mu', u', t') \text{ corresponding to } {}^{\Theta}T^*X , \end{equation}
on the left and right copies of the respective cotangent bundles. We fix 
\[ V =\frac{\rho}{\rho'} \quad W = \frac{w-w'}{\rho'}, \quad Z = \frac{z-z'}{(\rho')^2} \]
as coordinates valid near the front face $\bhs{F}$, away from $\beta^\#(\{ \rho'=0 \})$. The map \eqref{double-ident} gives an identification between the 1-forms
\[  \frac{\mu}{\rho}d\rho - \frac{\mu'}{\rho'}d\rho' + \frac{u}{\rho}dw - \frac{u'}{\rho'}dw' + \frac{t}{\rho^2}dz - \frac{t'}{(\rho')^2}dz' , \] 
and 
\[\alpha\; dV + \wt \xi d\rho' + \beta dW + \wt\kappa dw' + \gamma dZ + \wt \eta dz' ,   \]
defined on ${}^{\Theta}T^*X\times {}^{\Theta}T^*X$ and $T^*(X\times_\Theta X)$ respectively. We will first determine how the coefficients of these 1-forms are related under the map \eqref{double-ident}, in this neighborhood of $\bhs{F}$. Since $\rho=V\rho',\; w=w'+\rho'W,\; z=z'+(\rho')^2Z$ we have 
\[ d\rho = Vd\rho' + \rho'dV , \quad dw = dw'+\rho'dW+Wd\rho', \quad dz = dz' + 2\rho' Z d\rho' + (\rho')^2 dZ  \]
and so the canonical 1-form in $T^*(X\times_\Theta X)$ is given by,
\begin{align*} 
 \left(\frac{\mu}{\rho}V - \frac{\mu'}{\rho'} + \frac{u}{\rho}W + \frac{2t\rho'}{\rho^2}Z\right)&d\rho' + \frac{\mu \rho'}{\rho}dV + \left(\frac{u}{\rho}-\frac{u'}{\rho'}\right)dw'  +  \frac{u\rho'}{\rho}dW  + \left(\frac{t}{\rho^2} - \frac{t'}{(\rho')^2}\right)dz' + \frac{t(\rho')^2}{\rho^2}dZ \\
& = \alpha\; dV + \wt \xi d\rho' + \beta dW + \wt\kappa dw' + \gamma dZ + \wt \eta dz' 
\end{align*}
where
\begin{align*}
&  \alpha = \mu \frac{\rho'}{\rho}, \quad \beta = u \frac{ \rho'}{\rho}, \quad \gamma = t \left(\frac{\rho'}{\rho} \right)^2, \\
&\quad \; \wt\xi = \frac{\mu}{\rho'} - \frac{\mu'}{\rho'} + \frac{u}{\rho'} \frac{w-w'}{\rho} + \frac{2t}{\rho'}\frac{z-z'}{\rho^2} , \quad \wt\kappa = \frac{u}{\rho} - \frac{u'}{\rho'}, \quad \wt\eta = \frac{t}{\rho^2} - \frac{t'}{(\rho')^2}.
\end{align*}

Now, using the fact that $\chi$ is a $\Theta$-canonical relation (and thus ${}^{{\Theta}}\alpha_X - \chi^*({}^{\Theta}\alpha_X)=0$), and the fact that $\chi$ restricts to the identity over $\pa X$. To determine $\text{Gr}(\chi)$ in the coordinates \eqref{theta-coord} we observe first that $\rho$ and $\rho^\prime$ are both bdfs on $X$ and thus conformal: $\rho'=f \rho$. Further, we have that $\pi_X:{}^{\Theta}T^*X\to X$, and $(0,w',z'):=(\pi_X\circ \chi)|_{\pa X}(x,\zeta)=(0,w,z)$, hence $w'=w+\rho A,\; z'=z+\rho^2 B$ for some smooth functions $A,B$ on ${}^{\Theta}T^*X$. Finally, we use the relation between the fundamental 1-forms to observe that
\begin{align*}  \mu \frac{d\rho}{\rho} &+ u\frac{dw}{\rho} + t \frac{dz}{\rho^2} = \chi^*\left( \mu' \frac{d\rho'}{\rho'} + u'\frac{dw'}{\rho'} + t'\frac{dz'}{(\rho')^2} \right) \\
& = \mu' \left(\frac{d\rho}{\rho} + \frac{df}{f} \right) + \frac{u'}{a}\left( \frac{dw}{\rho} + dA + A \frac{d\rho}{\rho}\right) + \frac{t'}{a^2}\left( \frac{dz}{\rho^2} + 2B \frac{d\rho}{\rho} + dB \right)  \\
& = \left(\mu' + \frac{u'}{f}A + \frac{2t'}{f^2} \right)\frac{d\rho}{\rho} + \frac{u'}{f}\frac{dw}{\rho} + \frac{2t'}{f^2} \frac{dz}{\rho^2} + \left(\frac{\mu'}{f}df + \frac{u'}{f}dA +\frac{t'}{f^2}dB\right).
\end{align*}
The final bracketed term will only contribute terms which are $\cO(\rho)$ or $\cO(\rho^2)$ after computing their $\Theta$-differential (e.g. $df = \rho \pa_\rho f \tfrac{d\rho}{\rho} + \rho \pa_{w}f \frac{dw}{\rho} +\rho^2 \pa_zf \frac{dz}{\rho^2}$), thus after grouping such terms we obtain
\[ \mu \frac{d\rho}{\rho} + u\frac{dw}{\rho} + t \frac{dz}{\rho^2} =  \left(\mu' + \frac{u'}{f}A + \frac{2t'}{f^2} +\rho C \right)\frac{d\rho}{\rho} + \left( \frac{u'}{f} +\rho D\right) \frac{dw}{\rho} + \left(\frac{2t'}{f^2} +\rho^2E \right) \frac{dz}{\rho^2} , \]
where $f>0, A,B,C,D$ are smooth functions of $(\rho, w,z,\mu,u,t)$. Taken together, these computations imply that its graph is of the form
\begin{align*} \text{Gr}(\chi) 
& = \{ ((\rho,w, z,\mu,u,t) , (\rho',w', z', \mu',u',t')) \; | \; \rho' = f \rho, \; w' = w+\rho A, \; z'=z + \rho^2 B \; \\
& \mu'=\mu - u A - 2t B +\rho C, \; u' =  f u + \rho D ,\; t' = f^2 t + \rho^2 E   \}  
\end{align*}
From this we can see that 
\[ \alpha = e^f \mu, \quad \beta = e^f u, \quad \gamma = e^{2f} t, \quad \wt\xi = - e^{-f} C , \quad \wt\kappa = -e^{-f}D, \quad \wt\eta = -e^{-2f}E . \]
Since $e^{f} = \tfrac{\rho'}{\rho}$ is smooth and positive on $\Lambda_\chi$, the map \eqref{double-ident} (defined over the interior $\text{Int}(X\times X)$), extends by continuity to the boundary when restricted to $\text{Gr}(\chi)$, thus defining $\Lambda_\chi$
\[ \text{Gr}(\chi) \simeq \Lambda_\chi \hookrightarrow T^*(X\times_\Theta X)  \]
as the image of $\text{Gr}(\chi)$ under the map \eqref{double-ident}. Further, this shows that $\Lambda_\chi$ intersects the boundary of $T^*(X\times_\Theta X)$ only over $\bhs{F}=\{\rho'=0\}$ and does so transversely. Thus it is an extendible Lagrangian, and we have by the previous lemma that this intersection $\Lambda_{\chi_\Theta}$ is a Lagrangian submanifold of $T^*\bhs{F}$. 
\end{proof}
This lemma elucidates the name liftable canonical transformation as they provide examples of canonical transformation with ``good" lifts to $T^*(X\times_\Theta X)$ as the associated Lagrangian meets the diagonal only in the front face $\bhs{F}$.

Given $p\in \CI({}^{\Theta}T^*X)$ we define its $\Theta$-Hamiltonian vector field by the relation ${}^{\Theta}\omega(-,{}^{\Theta}H_p)=dp$. In local coordinates in which ${}^{\Theta}\omega$ is given by \eqref{symp-form-coord}, ${}^{\Theta}H_p$ is given by

\begin{align*}
{}^{\Theta}H_p = & \rho \frac{\pa p}{\pa \mu}\pa_\rho - \left( \rho\frac{\pa p}{\pa \rho} +w_j\frac{\pa p}{\pa w_j} + 2t\frac{\pa p}{\pa t} \right) \pa_\mu \\
& + \sum_{j=1}^{2n} \left(\rho \frac{\pa p}{\pa u_j} \right) \pa_{w_j} - \left( \rho \frac{\pa p}{\pa w_j} - u_j\frac{\pa p}{\pa \mu} \right)\pa_{u_j} + \rho^2\frac{\pa p}{\pa t}\pa z -\left( \rho^2 \frac{\pa p}{\pa z} - 2t \frac{\pa p}{\pa \mu} \right) \pa_t
\end{align*}
And observe that this vector field has the special property that the projection of the vector field to the base vanishes when restricted to $\pa X$.

Because our focus is the wave equation we are most interested in the Hamiltonian associated to our ACH metric. Since our metric satisfies
\[ g_{ACH}=\frac{d\rho^2 + h_\Theta(w,z,dw, dz)+ \rho Q(\rho,w,z,dw, dz)}{\rho^2}  \]
we can conclude its dual metric on $T^*X$ has the form
\[ G = (\rho \xi)^2 + \rho^2 h_\Theta(w, z, \eta_H, \eta_V) +\rho^3 Q(\rho, w,z, \eta_H, \eta_V) 
\]
or in the coordinates $(\rho,w,z, \mu,u,t)$ on ${}^{\Theta}T^*X$ our dual metric is given by
\begin{equation}\label{symbol} G = \mu^2 + h_\Theta(w,z, u,t) + \rho Q(\rho,w,z,u,t).  \end{equation}
This function on ${}^{\Theta}T^*X$ will be the Hamiltonian of interest in our study of the wave equation.

\begin{lemma}{$\Theta$-canonical flowouts}

Let $G\in \CI({}^{\Theta}T^*X)$ be the dual metric associated to the metric $g_{ACH}$, and let ${}^{\Theta}H_G$ be its $\Theta$-Hamilton vector field. For all $s>0$, the canonical transformation $\chi_s: {}^{\Theta}T^*X \to {}^{\Theta}T^*X$, given as the flow-out of the Hamiltonian,
\[ \chi_s(q): = \emph{exp}(s\; {}^{\Theta}H_G)(q) ,  \]
is a liftable canonical transformation. Thus the graph of $\chi_s$ defines a smooth extendible Lagrangian submanifold of $T^*(X\times_\Theta X)$. Further, the intersection
\[ \Lambda_{\bhs{F}}(s): = \Lambda_s\cap T^*(X\times_\Theta X)|_{\bhs{F}}  \]
is a smooth Lagrangian submanifold of $T^*\bhs{F}$ given by 
\[ \emph{exp}(s\; H_{G_\Theta})(T^*\bhs{F} |_{D_\Theta\cap \bhs{F}} ) = \Lambda_{\bhs{F}}(s)  \]
where $G_\Theta=\wt G\rvert_{\bhs{F}}$, the restriction to the front face of the lift of $G$ to $T^*(X\times_\Theta X)$.
\end{lemma}
\begin{proof}
Since the flow-out of a Hamilton vector field is always a canonical transformation, the first claim follows from the fact that only the projection onto the base vanishes. 
Thus we only to check the claim regarding $\Lambda_{\bhs{F}}(s)$. We can study the graph of $\chi_s$ after viewing $G\in \CI({}^{\Theta}T^*X\times {}^{\Theta}T^*X)$ as a function depending only on the second copy of ${}^{\Theta}T^*X$. 

On this space, we can write our canonical 1-form in the coordinates
\begin{align} \label{diff-form}
 \pi_1^*({}^{\Theta}\alpha) - \pi_2^*({}^{\Theta}\alpha) & : = \tfrac{\mu}{\rho}d\rho - \tfrac{\mu'}{\rho'}d\rho' + \tfrac{u}{\rho}dw - \tfrac{u'}{\rho'}dw' + \tfrac{t}{\rho^2} dz - \tfrac{t'}{(\rho')^2} dz' ,
\end{align}
thus we can write the Hamilton vector field of a function on this space with respect to this 1-form, with the same formula as we calculated above. In this case $\chi_s$ is the flow-out of the diagonal in ${}^{\Theta}T^*X\times {}^{\Theta}T^*X$ along the vector field ${}^{\Theta}H_G$. In these coordinates our length function is given by
\[  G = (\mu')^2 + h_\Theta(w',z', u',t') + \rho' Q(\rho',w' , z',u' ,t')  \]
and we can consider local coordinates near the front face, projective with respect to the left face:
\[ V = \frac{\rho'}{\rho}, \quad W = \frac{w' - w}{\rho} , \quad Z = \frac{z'-z}{\rho^2}  \]
with blow-down map
\begin{align*}
\beta: \; X\times_\Theta X \to X\times X \quad \quad (\rho,w,z, V, W, Z ) &  \mapsto (\rho,w,z, \rho',w',z') \\
& = (\rho,w,z,\rho V, w+\rho W, z+ \rho^2 Z ).
\end{align*}
The pullback of \eqref{diff-form} by $\beta$ is,
\[ ad\rho + a_fdV + b dw + b_f dW + cdz + c_fdZ,  \]
and in these coordinates we have that $\bhs{F}=\{\rho=0\}$ and the interior lift of the diagonal $D_\Theta$ is given by $D_\Theta=\{V=1, W=Z=0\}$. The lift of $p$ to $T^*(X\times_\Theta X)$ is given by
\begin{align*} \wt p &=  (a_fV)^2 + h_0(w+\rho W, z+\rho^2 Z, -V b_f , -V^2c_f) + \rho V Q(\rho V, w+\rho W, z+\rho^2 Z, -V b_f , -V^2c_f))  \\
& = (a_fV )^2 + V^2 h_0(w+ \rho W, z+\rho^2 W, b_f,c_f) + \rho V^3 Q(\rho V, w + \rho W,z+\rho^2 Z, b_f,c_f)
\end{align*}
where the functions $h_0, Q$ are $\cD_\rho^{\bbH}$-homogeneous of order 2 in the fiber variables.

Now we lift our symplectic form \eqref{difference form} to $T^*(X\times_\Theta X)$, and denote it by $\wt \omega$, ${}^{\Theta}H_G$ lifts to $H_{\wt G}$. In the coordinates 
\[ [(\rho,w,z, V,W,Z) \; ; (a,b,c,a_f, b_f,c_f)] \in T^*(X\times_\Theta X) 
 \]
our lifted Hamilton vector field has the form
\begin{align*}
H_{\wt G} & = \left(\frac{\pa \wt G}{\pa a}\pa_\rho - \frac{\pa \wt G}{\pa \rho}\pa_a\right) + \left( \frac{\pa \wt G}{\pa a_f}\pa_V - \frac{\pa \wt G}{\pa V}\pa_{a_f}  \right) \\
& + \sum_{j=1}^{2n} \left( \frac{\pa \wt G}{\pa b^j}\pa_{w_j} - \frac{\pa \wt G}{\pa w^j}\pa_{b_j} \right) + \sum_{j=1}^{2n} \left( \frac{\pa \wt G}{\pa b_f^j} \pa_{W_j} - \frac{\pa \wt G}{\pa W^j} \pa_{(b_f)_j} \right)  \\
&+ \left( \frac{\pa \wt G}{\pa c}\pa_{z} - \frac{\pa \wt G}{\pa z}\pa_{c} \right) + \left( \frac{\pa \wt G}{\pa c_f}\pa_{Z} - \frac{\pa \wt G}{\pa Z}\pa_{c_f} \right) \\
& = 2a_fV^2\pa_V -2V[a_f^2 + h_\Theta ]\pa_{a_f} - V^2\sum_j W_j \frac{\pa \wt G}{\pa w^j}  \pa_a - V^2 \sum_j \frac{\pa h_\Theta}{\pa w_j} \pa_{b_j} \\
&+ V^2 \sum_j \frac{\pa h_\Theta}{\pa b_f^j} \pa_{W_j} + V^2 \frac{\pa h_\Theta}{\pa c_f} \pa_Z   + \cO(\rho) ,
\end{align*}
thus $H_{\wt p}$ is smooth all the way down to $\bhs{F} = \{\rho=0\}$. Further, with respect to our coordinate transformation on $T^*(X\times_\Theta X)$ induced by the blow-down map $\beta$, we have that the diagonal 
\[  \{\rho=\rho',w=w',z=z',\mu=\mu',u=u', t=t'\} = ({}^{\Theta}T^*X)_{\diag} \subset {}^{\Theta}T^*X\times {}^{\Theta}T^*X \]
lifts to 
\[  \wt\cD_\Theta = \{ V=1, W=Z=a=b=c=0 \} \subset T^*(X\times_\Theta X). \]
Thus $\wt \cD_\Theta$ transversely intersects $T^*(X\times_\Theta X)|_{\bhs{F}}$ at $ \{ \rho=0,  V=1, W=Z=a=b=c=0 \}$. Finally we see that $H_{\wt p}$ projects down to $T^*\bhs{F}$ as 
\begin{align*} 2a_fV^2\pa_V & - 2V(a_f^2 + h_\Theta(w,z,b_f,c_f))\pa_{a_f} \\
& + 2V^2\sum_{i,j} h_\Theta^{ij}(w, z) \cdot b_f^{(j)} \frac{\pa}{\pa b_f^{(j)} }  +2V^2 h_\Theta^{0,0}(w, z) \cdot c_f \frac{\pa}{\pa c_f}  
\end{align*}
which is precisely the Hamilton vector field of $a_f^2 V^2 + V^2 h_\Theta(w,z,b_f ,c_f)$.
\end{proof}

\begin{remark}
Notice that because
\[ \wt G|_{\bhs{F}} = V^2(a_f^2 + h_\Theta(w,z,b_f,c_f)) =: G_\Theta  \]
the projection of the Hamilton vector field of $G$ to $T^*\bhs{F}$ is precisely the Hamilton vector field of the restriction of $G$ to $\bhs{F}$, with respect to the induced symplectic form on $T^*\bhs{F}$.
\end{remark}

In other words, for the Hamiltonian given by our length functional \eqref{symbol}, we have that: $\forall s>0$, the twisted graph of $\text{exp}(s\; {}^{\Theta}H_G): {}^{\Theta}T^*X\to {}^{\Theta}T^*X$ defines a Lagrangian submanifold $\Lambda(s) \subset T^*(X\times_\Theta X)$. Further this Lagrangian intersects the boundary only over $\bhs{F}$, and it does so transversely. The transversal intersection is itself a Lagrangian flow-out
\[ \Lambda_{F}(s):=\text{exp}(s\; H_{G_\Theta})(T^* \bhs{F}|_{D_\Theta \cap \bhs{F}}) \subset T^*\bhs{F} \]
which is the flow-out by the Hamilton vector field of $G_\Theta=\sigma_2( N_{\bhs{F}}(\Delta))$, the principal symbol of the normal operator at the front face. 

\subsection{$\Theta$-FIOs and the Wave Kernel}$ $\\

Here we construct the calculus of operators that our wave group $\cos(t\sqrt{\Delta - n^2/4})$ will lie in. These shall be restricted to the subclass of Lagrangian distributions whose support does not meet the left or right faces, $\beta^*(\pa X\times X)$, and $\beta^*(X\times \pa X)$ respectively. Due to the finite speed of propagation, initial data $U(t,p,p')$ supported in the interior of $\bhs{F}$ which evolves according to the wave equation,
\[ \begin{cases}  \left(D_t^2 + \Delta_g - \tfrac{n^2}{4}\right) U(t,p,p') = 0 \\ U(0,p,p') = \delta(p,p') ,\quad \pa_t U(0,p,p') = 0 \end{cases} \]
remains supported away from the left and right faces, $\bhs{L},\bhs{R}$. In particular, when considering our calculus of FIOs, we can ignore the complement of the front face in the corner, and restrict ourselves to Lagrangians which meet the boundary only at $\bhs{F}$. 

Since the canonical relation $C$ of the wave group will be a Lagrangian in $T^*\bbR\times T^*(X\times_\Theta X)$, we mildly extend our class of Lagrangians from the last section. The canonical 1-form on $T^*\bbR\times {}^{\Theta}T^*X\times {}^{\Theta}T^*X$ is given by
\[ \alpha = td\tau + \tfrac{\mu}{\rho}d\rho - \tfrac{\mu'}{\rho'}d\rho' + \tfrac{u}{\rho}dw - \tfrac{u'}{\rho'}dw' + \tfrac{s}{\rho^2} dz - \tfrac{s'}{(\rho')^2} dz.'  \]
With this 1-form, we can define a canonical relation
\[ C=\left\{ (t,\tau,\zeta_1,\zeta_2) \bigg|\, \tau + \sqrt{ G(\zeta_1,\zeta_1)} =0,  \; \zeta_2 = \text{exp}(t\; {}^{\Theta}H_G)(\zeta_1) \right\} \subset T^*\bbR\times {}^{\Theta}T^*X\times {}^{\Theta}T^*X, \]
and this canonical relation in turn defines a Lagrangian of $T^*\bbR \times T^*(X\times_\Theta X)$ given by
\[ \Lambda_C=\left\{ (t,\tau,\zeta_1,\zeta_2) \bigg|\, \tau + \sqrt{\wt G(\zeta_1,\zeta_2)} =0,  \; (\zeta_1,\zeta_2)\in \Lambda_t \right\} \subset T^*\bbR\times T^*(X\times_{\Theta}X) \]
where $\Lambda_t$ is an extendible Lagrangian associated to the graph of the liftable canonical transformation 
\[ \chi_t:=\text{exp}(t\; {}^{\Theta}H_G)(q): {}^{\Theta}T^*X\to {}^{\Theta}T^*X,\] 
and $\wt G$ is the lift of $G$ from the second copy of ${}^{\Theta}T^*X$. In particular, this Lagrangian intersects the boundary only over the front face $\bhs{F}$, and 
\[ \Lambda_C^\Theta = \left\{ (t,\tau, \bar\zeta_1, \bar\zeta_2)\bigg|\, \tau + \sqrt{G_\Theta(\bar\zeta_1,\bar\zeta_2)} =0,  \; (\bar\zeta_1,\bar\zeta_2)\in \Lambda_{\bhs{F}}(t) \right\} \subset T^*\bbR\times T^*\bhs{F}\]
where $G_\Theta$ is the restriction of $\wt G$ to $\bhs{F}$.


Now, given a liftable canonical transformation $\chi: {}^{\Theta}T^*X \to {}^{\Theta}T^*X$ we define our \textbf{$\Theta$-Fourier Integral Operators} associated to $\chi$ to be the linear operators $L: \cE'(X)\to \cD'(X)$ whose Schwartz kernels lie in the space of distributions
\begin{align*} 
I_\Theta^{m,s}(X; \chi, {}^{\Theta}\Omega^{1/2}) := \{& \rho_{F}^{s} \cK_L \big|\, \cK_L \in I^m(X\times_\Theta X; \Lambda_\chi, {}^{\Theta}\Omega^{1/2}), \; \rho_F^{s}\cK_L \text{ vanishes}\\
& \text{ in a neighborhood of }\pa(X\times_\Theta X)\setminus \bhs{F} \}  
\end{align*}
where $\Lambda_\chi$ is the extendible Lagrangian submanifold of $T^*(X\times_\Theta X)$ associated to $\chi$ by lemma \ref{lem:liftlagr}. Similarly, for the canonical relation $C$ defined above, we say that $\Theta$-Fourier Integral Operators associated to $C$ are the linear operators $B:\cE'(\bbR\times X)\to \cD'(X)$ whose Schwartz kernels lie in the space of distributions
\begin{align*} 
I_\Theta^{m,s}(\bbR\times X,X; C, {}^{\Theta}\Omega^{1/2}) := \{& \rho_{F}^{s}\cK_B \big|\, \cK_B \in I^m(\bbR\times X\times_\Theta X; \Lambda_C, {}^{\Theta}\Omega^{1/2}), \;  \rho_F^{s}\cK_B \text{ vanishes}\\
& \text{ in a neighborhood of }\pa(\bbR \times X\times_\Theta X)\setminus (\bbR\times \bhs{F}) \}  
\end{align*}
In both cases, such operators are those whose Schwartz kernels are Lagrangian distributions with respect to $\Lambda_\chi$, ($\Lambda_C$ resp.), and vanish to order $s$ at the front face $\bhs{F}$. Such operators carry two different principal symbol mappings: one is the usual symbol of a Lagrangian distribution, in the interior; the second operator is obtained by the principal symbol of the normal operator $\cK_L\rvert_{\bhs{F}}$ (resp. $\cK_{B}\rvert_{\bbR\times \bhs{F}}$) associated to the Lagrangian in $T^*\bhs{F}$ (resp. $T^*\bbR\times T^*\bhs{F}$). 

This second symbol is again the symbol of a Lagrangian distribution from the fact that our Lagrangian $\Lambda_\chi$ (resp. $\Lambda_C$) has transversal intersection with $T^*(X\times_\Theta X)\rvert_{\bhs{F}}$ (resp. $T^*\bbR\times T^*(X\times_\Theta X)\rvert_{\bhs{F}}$), thus  the restriction of  Lagrangian distribution to $\bhs{F}$ is again a Lagrangian distribution with respect to $\Lambda_\chi$ (resp. $\Lambda_C$). 

We now take a moment again to highlight the normal operator. If $\cK_A\in I_\Theta^{m,s}$, then $N_p(A) = (\rho_F^{-s} \cK_A)|_{\bhs{F}}$, and $N_p(A)$ is a Lagrangian distribution with respect to $\Lambda_\chi$ (resp. $\Lambda_C$). Further the normal operator satifies an analogue of the short exact sequence for principal symbols of operators:

\begin{proposition}\label{normal sequence}
The normal operator participates in a short exact sequence
\[ 0\to I_{\Theta}^{m,1}(\bbR\times X,X; C, {}^{\Theta}\Omega^{1/2}) \hookrightarrow I_{\Theta}^{m,0}(\bbR\times X,X; C, {}^{\Theta}\Omega^{1/2}) \xrightarrow{N_p(-)} I^m(\bbR\times \bhs{F}; \Lambda_C^\Theta, \Omega^{1/2}) \to 0 \]
such that for any $\Theta$-differential operator $P\in \emph{Diff}_\Theta^m(X)$ and any $\Theta$-\emph{Fourier integral operator} $B\in I_\Theta^{m,s}(\bbR\times X,X; C, {}^{\Theta}\Omega^{1/2})$ we have
\[ N_p((D_t^2-P)\circ B) = (D_t^2 - N_p(P))\ast N_p(B) \]
\end{proposition}
\begin{proof}
This is an analogue of \cite[prop 5.19]{MazMel}, and \cite[prop 3.1]{Jo-Sa2}. 

The injectivity portion of the statement of exactness is immediate from the definition. Since we have $N_{\hat{p}}(-)$ is $\CI$ in $\hat{p}\in \pa X$ and defines an operator on $\CI(\bhs{F_{\hat{p}}})$ for each $\hat p$ fixed. In particular, since the kernels of these operators are smooth up to the front face, it makes sense to consider their Taylor series on $\bhs{F_{\hat{p}}}$. The surjectivity of $N_{\hat{p}}(-)$ thus arises from a version of Borel's lemma for the Taylor series of $N_{\hat{p}}(-)$ in local coordinates for $\bhs{F_{\hat{p}}}$.

To prove the composition formula, we can use the structure of the Normal operator at $\bhs{F}$, and the fact that we are not blowing up in the $t$ variable, so it commutes with the normal operator. 

We observe first that such a $P\in \text{Diff}_{\Theta}^m(X)$ can be written with respect to our frame $\{\rho \pa_\rho, \rho V_{w_i}, \rho^2 \pa_z\}$ for ${}^{\Theta}TM$:
\begin{align*} &P=\sum_{j+|\alpha|+k\leq m} a_{j\alpha k}(\rho,w,z) (\rho\pa_\rho)^j (\rho V_w)^\alpha (\rho^2 \pa_z)^k \\ \implies & N_{\hat{p}}(P) = \sum_{j+|\alpha|+k\leq m} a_{j\alpha k}(0,w,z) (\rho\pa_\rho)^j (\rho V_w)^\alpha (\rho^2 \pa_z)^k. 
\end{align*}
As usual we choose to identify this as acting on $1/2$-densities: if we choose coordinates $(\rho,w,z)$, these induce a trivialization of the square root of the $\Theta$-density bundle $\Omega_\Theta^{1/2} = \Omega^{1/2}$ 
\[ \gamma= (\rho)^{-(2n+3)} \left| d\rho dw dz \right|^{1/2} \]
and $P$ acts on $f\in \CI(X; \Omega_\Theta^{1/2})$ by $Pf = P(f \gamma^{-1})\gamma$. Of course this is simply for $\Theta$-differential operators. More generally $\Theta$-FIOs will act on $1/2$-densities via their normal operator: $N_{\hat{p}}(A) = (\rho_F^{-s} \cK_A)|_{\bhs{F_{\hat{p}}}} $
\[ (B_tf)(\rho,w,z)\cdot \gamma = \int_{\bhs{F_p}} \cK_{B_t}(0,w, z;V,W,Z)f\left(\frac{\rho}{V},w-\frac{\rho}{V}W,z-\left(\frac{\rho}{V}\right)^2 Z\right) \frac{dV dWdZ}{V}\cdot \gamma  \]
In particular, this implies that the normal operator of $P\circ B$
\[ N_{\hat{p}}(P\circ B) = \left( \sum_{j+|\alpha|+k\leq m} a_{j\alpha k}(0,w,z) (\rho\pa_\rho)^j (\rho V_w)^\alpha (\rho^2 \pa_z)^k \circ \left( \cK_{B_t}(0,w, z;V,W,Z) \right)  \right)\cdot \gamma  \]
\end{proof}

Having proven this lemma, we arrive at a short time parametrix for the wave group.

\begin{proposition}\label{elmt-alg}$ $\\
For each $t\in \bbR$, for the canonical relation 
\begin{align*} C= \{  [ (t,\tau),& (\rho,w,z;  \mu,u,s) , (\rho',w',z'; \mu', u' ,s') ] : \\
& \tau + \sqrt{G (\rho,w,z; \mu,u,s)}=0 \; ; \; \;   (\rho',w',z'; \mu', u' ,s') =  \emph{exp}(t {}^{\Theta}H_G)(\rho,w,z;  \mu,u,s) \}  \end{align*}
the wave group $U(t)$ is $\Theta$-Fourier integral operator of the class
\[ U(t) = \cos\left(t\sqrt{\Delta_g - n^2/4}\right) \in I_\Theta^{-1/4,0}(\bbR\times X, X; C, {}^{\Theta}\Omega^{1/2} )  \]
\end{proposition}
\begin{proof}
Given the normal sequence, the argument reduces to a purely local one: using proposition \ref{normal sequence}, and the fact that
\[ N_{\hat{p}}(\Id) = \delta(V-1)\delta(W)\delta(Z)\gamma = \delta(0_p)\gamma  \]
we can take as ansatz $U_0(t,\hat p)=N_{\hat{p}}(U(t))$ the wave group in this fiber $\bhs{F_{\hat{p}}}$: 
\[  \begin{cases}  \left(D_t^2 + N_{\hat{p}}(\Delta_g) - \tfrac{n^2}{4}\right) U_0(t,\hat p) = 0 \\ U(0,\hat p) = \delta(0_{\hat p}) ,\quad \pa_t U(0,\hat p) = 0 \end{cases}  \]
here $0_{\hat{p}}\in \bhs{F_{\hat{p}}} \simeq \bbX_{\hat{p}}$ corresponds to the identity element in the group. Note also that the specific form of the model Laplacian
\[ N_{\hat{p}}(\Delta_g) = -\tfrac{1}{4}(\rho\pa_\rho)^2 + \tfrac{n+1}{2}\rho\pa_\rho + \rho^2 \Delta_{H}(\hat{p}) - \rho^4Z^2(\hat{p})  \]
means we can also construct the model wave group, and study its asymptotics via analyzing those of the wave group in $\bbH_{\bbC}^{n+1}$.

Since $0_{\hat{p}}\in \text{Int}(\bhs{F_{\hat{p}}})$, it does not meet the corners of $\bhs{F_{\hat{p}}}$. Similarly $\Lambda_C$ does not meet the corners in finite time, so we can follow the argument of Duistermaat-Guillemin prop 1.1 to conclude $U_0(t)\in I^{-1/4}(\bbR\times \bhs{F}; \Lambda_C^\Theta, \Omega)$

Now we iterate. Choose a $u_0\in I_\Theta^{-1/4,0}(X\times \bbR, X; C, {}^{\Theta}\Omega^{1/2} )$ such that $N_{\hat{p}}(u_0)=U_0(t)$. Then 
\[ \beta_L^*(D_t^2+ \Delta_g - n^2/4)(U(t)-u_0) = r_0\in  I_\Theta^{-1/4,1}(X\times \bbR, X; C, {}^{\Theta}\Omega^{1/2} )  \]
and $\rho^{-1}r_0\in I^{-1/4,0}$ where $\rho$ is a defining function for the left face. (This is well-defined since $r_0$ is supported away from the left face, as $u_0$ was, and the wave operator preserves this support due to the condition on wave front of $U_0$, via \cite[Thm 2.5.15]{Hor3}). Now we solve the inhomogeneous wave equation to find a $u_1\in I_\Theta^{-1/4,0}$ solving
\[ \begin{cases}  \left(D_t^2 + N_{\hat{p}}(\Delta_g) - \tfrac{n^2}{4}\right) N_{\hat{p}}(u_1) = N_{\hat{p}}(\rho^{-1}r_0) \\ N_{\hat{p}}(u_1)|_{t=0} = N_{\hat{p}}(\rho^{-1}r_0) ,\quad \pa_t N_{\hat{p}}(u_1)|_{t=0} =  \pa_t N_{\hat{p}}(\rho^{-1}r_0)|_{t=0} \end{cases}   \]
solving as before we obtain such a $u_1$. We now have $\beta_L^*(D_t^2+ \Delta_g - n^2/4)(U(t)-u_0 - \rho u_1) = r_1\in I_\Theta^{-1/4,-2}$. 

Proceeding iteratively we obtain $U_\infty \sim \sum_{j\geq 0} \rho^j u_j$ such that $\beta_L^*(D_t^2+ \Delta_g - n^2/4)U_\infty$ vanishes to infinite order at $\bhs{F}$. The error term also has infinite order vanishing at $\bhs{F}$ in the Cauchy data from the construction. Finally, after extending this error term to be identically zero across the front face, we can use H\"{o}rmander's transverse intersection calculus to remove this error term (see e.g. \cite[Thm 2.5.15]{Hor3}).
\end{proof}

Unfortunately, this is a short time parametrix, as this construction is only valid for finite $t$. If we allow $t\to \infty$, our Lagrangian flow-out $\Lambda(t)$ will meet the corners of $\bhs{F}$, which would require a more sophisticated composition formula.

\section{Wave Trace Asymptotics}\label{sec:Wavetrace}

Now that we know the wave group is a $\Theta$-Fourier integral operator we can ask whether its trace can be studied, as in the case of the wave trace on a compact manifold without boundary. This presents some technical difficulties, since the operator $\cos\left(t\sqrt{\Delta_g-\tfrac{(n+1)^2}{4}}\right)$ is not trace class, so we need to introduce a regularization of its trace. 

Heuristically, our goal is to study the trace, 
\begin{equation}\label{eqn:wavetrace}
    \Tr U(t)=\int_{\bar X_{\diag}} U(t,x,x) \; d\Vol_g = \Pi_\ast \iota_{\diag}^*U  
\end{equation} 
using appropriate maps $\iota_{\diag}, \Pi$, to define this integral via pullback and pushforward. An analysis of the wavefront sets of these maps will permit an analysis of their associated operators, and prove that the resulting object is well-defined distribution on $\bbR$, with wavefront set to be determined. 

First, notice that for all $p,p^\prime\not\in \pa \bar X$, the restriction of $U(t,p,p^\prime)$ to the diagonal $\bar X_{\diag}$ is well-defined. To see this we proceed as in \cite[\S 1]{Du-Gu} by introducing the map,
\[ \iota_{\diag}:\bbR\times \bar X_{\diag} \to \bbR\times \bar X\times \bar X, \quad (t,p)\mapsto (t,p,p)  \]
of the inclusion of the diagonal. Pullback along this map is a Fourier integral operator of order $\tfrac{n+1}{2}$, defined by the canonical relation
\[ \wfp(\iota_{\diag}^*) = \left\{ \big(((t,\tau) , (p,\zeta+\zeta^\prime)\big) ,\; \big((t,\tau),(p,\zeta),(p,\zeta^\prime)) \big)  \right\} = N^*\{\iota_{\diag}(t,p)=(t,p,p^\prime)\} . \]
Now, using the fact that $\wf(U)=C$ (as defined in proposition \ref{elmt-alg}), assuming $p,p^\prime\not\in \pa \bar X$, then whenever $((t,\tau), (p,\zeta), (p,\zeta^\prime))\in \wfp(U)$ we have $\tau\neq 0$, thus $\wf(U)\cap N_{\iota_{\diag}}=\emptyset$ at such points (where $N_{\iota_{\diag}}=\{ (\iota(t,p),\tau,\zeta,\zeta^\prime)\in T^*(\bbR\times \bar X\times \bar X): D\iota_{\diag}^\intercal (\tau,\zeta,\zeta^\prime)=0 \}$ is the set of normals of the map). Thus we can apply \cite[Thm $2.5.11^\prime$]{Hor3} to conclude that $\iota_{\diag}^*U$ is a well-defined distribution on $\bbR\times (\bar X\setminus \pa\bar X)$ with wavefront set
\[ \wfp(\iota_{\diag}^*U) = \{ ((t,\tau),(p,\zeta-\zeta^\prime)): \tau+\sqrt{G(p,\zeta)}=0, \; (p,\zeta)=\text{exp}(t{}^{\Theta}H_G)(p,\zeta^\prime) \} .  \]
Duistermaat-Guillemin next study the wavefront set of the projection $\Pi: \bbR\times \bar X\to \bbR$. In our case we now introduce the regularization procedure. For $\eps>0$, define $\bar X_{\eps}=\{\rho>\eps\}$ for our bdf $\rho$. Consider the cutoff projection
\[ \Pi_\eps: \bbR\times \bar X_\eps\to \bbR, \quad (t,p)\mapsto t ,  \]
for which integration over the range $p$ is equal to the pushforward along $\Pi_\eps$ (the transpose of the operator $\Pi^*$). This map thus defines a Fourier integral operator of order $\tfrac12-\tfrac{n+1}{2}$ defined by the canonical relation
\[ \wfp(\Pi_*) = \left\{ \big((t,\tau),\; ((t,\tau),(x,0))\big) \right\} . \]
Again applying H\"{o}rmander's Theorem \cite[Thm $2.5.11^\prime$]{Hor3} we can conclude that the cutoff wave trace
\[ T_\eps(t) = \int_{\rho>\eps} U(t,p,p) = (\Pi_\eps)_*(\iota_{\diag}^*U(t)) \]
is a well-defined distribution on $\bbR$ satisfying
\[ \wf(T_\eps(t)) = \{ (t,\tau): \tau<0 \; \text{and }(p,\zeta)= \text{exp}(t\; {}^{\Theta}H_G)(p,\zeta^\prime) \text{ for some } (p,\zeta), \; \rho(p)>\eps \} .  \]
We obtain as a corollary
\begin{corollary}\label{singsupp-cor}
For $\eps>0$, the singular support of $T_\eps\in \cD^\prime(\bbR)$ is contained in the set of periods of closed geodesics in $\bar X_\eps$. Moreover, there exists $\eps_0>0$ such that all closed geodesics of $(X,g)$ with period greater than zero are contained in $\bar X_{\eps_0}$.

In particular for all $\eps<\eps_0$, the singular support of $T_\eps$ is contained in the set of period of closed geodesics of $X$.
\end{corollary}
\begin{proof}
Only the claim regarding closed geodesics remaining in $\bar X_{\eps_0}$ remains to be proven. This is a statement about strict convexity of the geodesic flow in a neighborhood of infinity (see e.g. \cite[Proposition 4.1]{Jo-Sa2}, \cite[Lemma 4.1]{DaVa}). We show that if $\eps$ sufficiently small, any geodesic $\gamma$ which intersects $\{\rho<\eps\}$ cannot be closed. Introducing coordinates $(\rho, w, z)$ with corresponding dual coordinates $(\xi,\eta_H,\eta_V)$, such that $\rho$ is a boundary defining function for $\pa \bar X$. 

In these coordinates, we write the metric in a collar neighborhood of the boundary as 
\[ g = \frac{4d\rho^2 + \wt g_\rho}{\rho^2}, \quad \wt g_\rho = h_\sH+ \rho^{-2} \theta^2  \]
and we write 
\[ G_\rho(\eta,\eta) = h_\sH(\eta_H,\eta_H) + \rho^2 \theta^2(\eta_V,\eta_V)  \]
for the bilinear form on $T^*X$ induced by the dual metric of $\wt g_\rho$. In these coordinates the geodesic Hamiltonian is given by
\[ |\zeta|_g^2 = \sigma^2 + \bar G(\mu,\mu) = \sigma^2 + h_\sH(\mu_H,\mu_H) + \theta^2(\mu_V,\mu_V)  \]
where $\sigma=\rho\xi, \mu_H=\rho \eta_H, \mu_V=\rho^2 \eta_V$, and $\bar G=\rho^2 G_\rho$. The Hamilton vector field of this function is given by
\[ H_{|\zeta|_g^2} = \pa_\xi |\zeta|^2\pa_\rho - \pa_\rho |\zeta|^2\pa_\xi + (\pa_{\eta_H}|\zeta|^2)\cdot Y - (Y|\zeta|^2)\cdot \pa_{\eta_H} + (\pa_{\eta_V}|\zeta|^2)\pa_z - (\pa_z|\zeta|^2)\pa_{\eta_V}  \]
where $\{Y_j\}_{j=1}^{2n}$ is a local $h_\sH$-orthonormal frame dual to $\{d\eta_H^j\}_{j=1}^{2n}$. Computing the change in these vector fields with respect to the change of coordinates $(\rho, w, z,\xi,\eta_H,\eta_V)\mapsto (\rho, w, z, \sigma, \mu_H,\mu_V)$ gives
\[ \pa_\xi = \rho\pa_\sigma, \; \pa_{\eta_H}=\rho\pa_{\mu_H}, \; \pa_{\eta_V} = \rho^2 \pa_{\mu_V}, \]
\[ \pa_\rho = \pa_\rho + \rho^{-1}\sigma\pa_\sigma + \rho^{-1}(\mu_H \cdot \pa_{\mu_H}+2 \mu_V \pa_{\mu_V}) , \; Y=Y, \; \pa_z = \pa_z. \] 
Thus the Hamilton vector field can be re-expressed as
\begin{align*}
H_{|\zeta|^2} &= (\rho \pa_\sigma |\zeta|^2)(\pa_\rho + \rho^{-1}\cR_{CC}) - (\rho\pa_\rho + \cR_{CC})(|\zeta|^2)\pa_\sigma + \rho[(\pa_{\mu_H}|\zeta|^2)\cdot Y - (Y |\zeta|^2)\cdot \pa_{\mu_H}] \\
&\quad \quad + \rho^2[(\pa_{\mu_V}|\zeta|^2)\pa_z-(\pa_z|\zeta|^2)\pa_{\mu_V}]
\end{align*}
where we have defined $\cR_{CC}=\mu_H \cdot \pa_{\mu_H}+2 \mu_V \pa_{\mu_V}$, the infinitesimal generator of the Heisenberg dilation action on $T^*\pa \bar X$. Using the facts that
\[ \pa_\sigma |\zeta|^2 = 2\sigma, \quad \cR_{CC}|\zeta|^2 = 2\bar G(\mu,\mu),  \]
and writing the vector field $H_{\wt g_\rho}=[(\pa_{\mu_H}|\zeta|^2)\cdot Y - (Y |\zeta|^2)\cdot \pa_{\mu_H}] + \rho [(\pa_{\mu_V}|\zeta|^2)\pa_z-(\pa_z|\zeta|^2)\pa_{\mu_V}]$, we can re-express this formula as
\[ H_{|\zeta|^2} = 2\sigma \rho\pa_\rho + 2\sigma \cR_{CC} - (2\bar G(\mu,\mu)+ \rho\pa_\rho\bar G)\pa_\sigma + \rho\cdot H_{\wt g_\rho}.   \]
Thus, along integral curves of the vector field $H_{|\zeta|^2}$ we have $\dot\rho = 2\sigma \rho, \dot\tau = -(2\bar G+\rho\pa_\rho \bar G)$. Thus, at a critical point of $\rho$ along the flow which is an interior point of $X$ we have
\[ \dot\rho=0 \implies \sigma=0,  \]
hence at such points we have
\[ \ddot\rho = 0 + 2\dot\sigma \rho = -2\rho(2\bar G+ \rho\pa_\rho\bar G) = -4\rho\bar G - 2\rho^2 \pa_\rho \bar G.  \]
Now, using the fact that $\bar G|_{\rho = 0}$ is positive definite, thus for sufficiently small $\rho$ this quantity is negative. Thus we have shown that for all geodesic curves $\gamma$ which intersect $\{\rho\leq \eps\}$ satisfy,
\[ \dot\rho\circ \gamma =0 \implies \ddot\rho\circ \gamma < 0.  \]
Now, assuming for the sake of contradiction that $\gamma$ is closed. Then there exists $\delta \in (0,\eps)$ such that $\gamma$ intersects $\{\rho=\delta\}$ in at least two points. Therefore there exists a $s_0$ with $\rho\circ\gamma(s_0)>0$ where $\rho\circ\gamma$ has a minimum. However at such a minimum we have $\dot\rho\circ\gamma(s_0)=0$ and $\ddot\rho\circ\gamma(s_0)>0$, contradicting our convexity statement. 
\end{proof}

Using this corollary, we can now begin an analysis of the renormalized wave trace. If we denote by $u_j\in I_\Theta^{-1/4, j}(\bbR\times X,X;C, {}^{\Theta}\Omega^{1/2})$ be the operators defined in the proof of proposition \ref{elmt-alg}. The same arguments used above can be used to show that the distribution
\[ I_j(t,\eps) = \int_{\rho>\eps} \rho^j u_j(t,p,p)    \]
is well-defined, with singular support satisfying the conclusions of corollary \ref{singsupp-cor}. Since $\bhs{F}$ and $\Lambda_C$ intersect transversally, only the density factor implicit in this operator can obstruct the convergence of $I_j(t,\eps)$ as $\eps\to 0$. Since this density, a trivialization of the ${}^{\Theta}\Omega^{1/2}$-bundle, diverges at the rate $\rho^{-(2n+3)}$ at $\pa \bar X$, the integrals $I_j(t,\eps)$ converges for any $j\geq 2n+3$. Applying Taylor's Theorem to $u_j(t,p,p)$ as $\rho\to 0$, we see that there exists constants $C_j$ such that the limit
\[ {}^{R}\Tr U(t) = \lim_{\eps\to 0} \left[ \int_{\rho>\eps} U(t,p,p) - \sum_{j=-2n-2}^{-1} C_j \eps^j + C_0 \log(\tfrac{1}{\eps}) \right]  \]
exists, which we call the \emph{renormalized wave trace}. From corollary \ref{singsupp-cor}, we immediately obtain
\begin{proposition}
The singular support of ${}^{R}\Tr U(t)$ is contained in the set of periods of closed geodesics of $(X,g)$.
\end{proposition}

Finally, we can begin our analysis of the renormalized wave trace as $t\to 0$ (in fact its inverse Fourier transform). First we choose a cutoff function $\chi\in \CI_c(\bbR)$, with the appropriate support to study the transform of the cutoff wave trace. If we denote the first non-zero period of a closed geodesic on $(X,g)$ as $t_0$, then choose $\chi$ such that $\chi(t)=1$ for $|t|>\tfrac{t_0}{2}$ and $\chi(t)\equiv 0$ for $|t|>\tfrac{2t_0}{3}$.

Now, using the arguments of \cite{Hor1}, (which are purely local, applying to any paracompact manifold), or alternatively the proof of \cite[Prop 2.1]{Du-Gu}, we immediately obtain
\begin{proposition}
There exists coefficients $\{w_k\}_{k\in \bbN_0}\subset \bbR$ such that the cutoff wave trace $T_\eps(t)$ satisfies,
\begin{equation}\label{wave-trace-short-time}
\int_\bbR T_\eps(t) \chi(t) e^{t\mu}  dt \sim \frac{1}{(2\pi)^{2n+2}} \sum_{k=0}^\infty w_k \mu^{2n+2-2k} ,  \end{equation}
as $\mu\to 0$ and rapidly decaying as $\mu\to -\infty$. The leading term, $\omega_0=\Vol_g(\bar X_\eps)$
\end{proposition}

Given this result for the asymtotics of the cutoff wave trace $T_\eps(t)$ we can then conclude similarly for the full wave trace \ref{eqn:wavetrace} that
\begin{theorem}\label{main}
There exists coefficients $\{\omega_k\}_{k\in \bbN_0}\subset \bbR$ such that the renormalized trace ${}^{R}\Tr U(t)$ satisfies,
\begin{equation*}
\int_\bbR {}^{R}\Tr U(t) \chi(t) e^{t\mu}  dt \sim \frac{1}{(2\pi)^{2n+2}} \sum_{k=0}^\infty \omega_k \mu^{2n+2-2k} , \end{equation*}
as $\mu\to 0$ and rapidly decaying as $\mu\to -\infty$. The leading term, $\omega_0={}^{R}\Vol_g(X)$, is called the renormalized volume, and can be computed as
\begin{equation}
{}^{R}\Vol_g(X) = \lim_{\eps\to 0} \left[ \int_{\{\rho>\eps\}} d\Vol_g - \sum_{j=-2n-2}^{-1} d_j \eps^j - d_0\log(\tfrac{1}{\eps}) \right] ,  
\end{equation}
where the $d_j$ are the unique real numbers such that this limit exists. 
\end{theorem}


\nocite{*}


\begin{thebibliography}{CdVHT18}

\bibitem[AlBaNa20]{Ale-Ba-Na}
Spyridon Alexakis, Tracey Balehowsky, and Adrian Nachman. \emph{Determining a Riemannian Metric from Minimal Areas.} Advances in Mathematics. 2020 ; Vol. 366.

\bibitem[AlMa10]{Ale-Maz}
Spyridon Alexakis and Rafe Mazzeo. \emph{Renormalized Area and Properly Embedded Minimal Surfaces in Hyperbolic 3-Manifolds.} Comm. Math. Phys. 297 (2010) 621-651.

\bibitem[An10]{And}
Michael Anderson. \emph{Boundary regularity, uniqueness and non-uniqueness for AH Einstein metrics on 4-manifolds.} Advances in Mathematics. Volume 179, Issue 2, (2003), Pages 205-249.

\bibitem[Br95]{Branson}
Thomas Branson. \emph{Sharp inequalities, the functional determinant, and the complementary series.} Trans. AMS 347 (1995), 3671-3742.

\bibitem[ChDeLeSk05]{CDLS}
Piotr Chrusciel, Erwann Delay, John M. Lee, and Dale N. Skinner. \emph{Boundary Regularity of Conformally Compact Einstein Metrics.} J. Differential Geometry. 69 (2005) 111-136.
  
\bibitem[DaVa12]{DaVa}
Kiril Datchev and Andras Vasy. \emph{Gluing semiclassical resolvent estimates via propagation of singularities.} International Mathematics Research Notices, Volume 2012, Issue 23, 2012, Pages 5409-5443. 

\bibitem[DuGu75]{Du-Gu}
J.J. Duistermaat, Viktor Guillemin. \emph{The spectrum of positive elliptic operators and periodic bicharacteristics.} Invent. Math. \textbf{29}, 39-79 (1975).

\bibitem[DuH\"{o}72]{DuHo}
 J. J. Duistermaat, L. H\"{o}rmander. \emph{Fourier integral operators, II}. Acta Math., 128(3-4):183-269, 1972.

\bibitem[Ei68]{eich}
Martin Eichler. \emph{A new proof of the Baker-Campbell-Hausdorff formula.} Journal of the Mathematical Society of Japan. 20 (1968): 23-25.

\bibitem[EpMeMe91]{EMM}
Charles Epstein, Richard Melrose, and Gerardo Mendoza. \emph{Resolvent of the Laplacian on Strictly Pseudoconvex Domains.} Acta. Math., 167 (1991), 1-106

\bibitem[FeGr85]{Feff-Gra}
Charles Fefferman and C. Robin Graham. \emph{Conformal invariants.} \'{E}lie Cartan et les Math\'{e}matiques d'Aujourdui, Asterisque (1985), 95-116.

\bibitem[FeHi03]{Feff-Hi}
Charles Fefferman and Kengo Hirachi. \emph{Ambient metric construction of Q-curvature in conformal and CR geometries} Math. Res. Lett. 10 (2003), no. 5-6, 819--832

\bibitem[GrWi99]{Gra-Wit}
C. Robin Graham and Edward Witten. \emph{Conformal anomaly of submanifold observables in AdS/CFT correspondence.} Nuclear Physics B. Vol. 546, Issues 1-2, (1999), 52-64

\bibitem[GrZw01]{Gra-Zwo}
C. Robin Graham and Maciej Zworski. \emph{Scattering Matrix in Conformal Geometry.}  S\'{e}minaire \'{E}quations aux d\'{e}riv\'{e}es partielles (Polytechnique) (2000-2001), Talk no. 22, 14 p. 

\bibitem[Gr01]{Grieser}
Daniel Grieser, \emph{Basics of the {$b$}-calculus}, Approaches to singular
  analysis ({B}erlin, 1999), Oper. Theory Adv. Appl., vol. 125, Birkh\"{a}user,
  Basel, 2001, pp.~30--84.

\bibitem[Gu05]{Gui}
Colin Guillarmou. \emph{Absence of resonance near the critical line on asymptotically hyperbolic spaces.} Asymptot. Anal. 42 (2005), no. 1-2, 105--121

\bibitem[GuSa06]{Gu-SaBr}
Colin Guillarmou and Ant\^onio S\'a Barreto. \emph{Scattering and Inverse Scattering on ACH Manifolds.} Journal f\"{u} die reine und angewandte Mathematik (Crelles Journal) 622 (2006)

\bibitem[HeSk98]{skenderis}
Mans Henningson and Kostas Skenderis. \emph{The holographic Weyl anomaly.} J. High Ener. Phys. 07 (1998), 023, \url{hep-th/9806087} 

\bibitem[HMM17]{HMM}
Kengo Hirachi, Taiji Marugame, Yoshihiko Matsumoto. \emph{Variation of Total Q-prime Curvature}. Adv. Math. Vol. 306 (2017) 1019-145. 

\bibitem[H\"{o}68]{Hor1}
Lars H\"{o}rmander. \emph{The spectral function of an elliptic operator.} Acta Math. 121 (1968),
193218.

\bibitem[H\"{o}r71]{Hor3}
\bysame. \emph{Fourier integral operators I.} Acta Math. 127 (1971), 79183.

\bibitem[JoS\'{a}00]{Jo-Sa1}
Mark Joshi, Ant\^onio S\'a Barreto. \emph{Scattering and Inverse Scattering on Asymptotically Hyperbolic Manifolds.} Acta. Math., 184 (2000), 41-86.

\bibitem[JoS\'{a}01]{Jo-Sa2}
\bysame. \emph{Wave Group on Asymptotically Hyperbolic Manifolds.} Journal of Functional Analysis. 184, (2001) 291-312. 

\bibitem[Ma16]{Yo1}
Yoshihiko Matsumoto. \emph{GJMS operators, Q-curvature, an obstruction of partially integrable CR manifolds.} Diff. Geom. and Appl. Vol 45, 

\bibitem[Ma18]{Yo2}
Yoshihiko Matsumoto. \emph{Einstein metrics on strictly pseudoconvex domains from the viewpoint of bulk-boundary correspondence.} In: Byun J., Cho H., Kim S., Lee KH., Park JD. (eds) Geometric Complex Analysis. Springer Proceedings in Mathematics \& Statistics, vol 246. Springer, Singapore. \url{https://doi.org/10.1007/978-981-13-1672-2_18}

\bibitem[MaMe87]{MazMel}
Rafe Mazzeo and Richard Melrose. \emph{Meromorphic extension of the resolvent on complete spaces with asymptotically constant negative curvature.} J. Funct. Anal. 75 (1987), 260310.

\bibitem[Me94]{MelSpec}
\bysame, \emph{Spectral and scattering theory for the Laplacian on asymptotically Euclidian spaces}, in Spectral and scattering theory (M. Ikawa, ed.), Marcel Dekker, 1994.

\bibitem[Mel96]{Melrose:Corners}
\bysame, \emph{Differential analysis on manifolds with corners}, Available
  online at \url{http://www-math.mit.edu/~rbm/book.html}, 1996.

\bibitem[PeHiTa08]{PHT}
Peter Perry, Peter Hislop, and Siu-Hung Tang. \emph{CR-Invariants and Scattering Operator for Complex Manifolds with Boundary.} Analysis \& PDE. Vol 1, No. 2 (2008)

\bibitem[Va17]{Va}
Andras Vasy. \emph{Microlocal Analysis of Asymptotically Hyperbolic Spaces and High Energy Resolvent Estimates.} Adv. Math. Vol. 306 (2017) 1019-145


\end{thebibliography}
\end{document}